\documentclass[preprint]{elsarticle}
\usepackage [english]{babel}
\usepackage[utf8]{inputenc}
\usepackage{amssymb}
\usepackage{amsbsy}
\usepackage{amsmath,stmaryrd}
\usepackage{amsfonts}
\usepackage{booktabs}
\usepackage{graphicx}
\usepackage{url}
\usepackage{algorithm}
\usepackage{parskip}
\usepackage{makecell}
\usepackage{color}
\usepackage{xifthen}
\usepackage{hyperref}
\usepackage{listings}

\journal{arXiv}
\newcommand{\ol}{\overline}
\newcommand{\ul}{\underline}
\newcommand{\mc}{\mathcal}

\newcommand{\mcN}{\mc{N}}

\newcommand{\nn}{\textnormal{\textbf{n}}}

\newcommand{\dv}{{--}}
\newcommand{\tabnr}[2]{\mbox{\textnormal{// Table~\ref{#1}~(#2)}}}

\newtheorem{theorem}{Theorem}
\newtheorem{lemma}[theorem]{Lemma}
\newtheorem{corollary}[theorem]{Corollary}

\newtheorem{assumption}[theorem]{Assumption}

\newdefinition{remark}{Remark}
\newproof{proof}{Proof}

\begin{document}

\begin{frontmatter}

\title{Fast multigrid solvers for conforming and non-conforming multi-patch Isogeometric
Analysis}

\author[ricam]{Stefan Takacs}
\ead{stefan.takacs@ricam.oeaw.ac.at}

\address[ricam]{Johann Radon Institute for Computational and Applied Mathematics (RICAM),\\
Austrian Academy of Sciences}

\begin{abstract}
	Isogeometric Analysis is a high-order discretization method for 
	boundary value problems that uses a number of degrees of freedom which is as small as for
	a low-order method. Standard isogeometric discretizations require a global parameterization
	of the computational domain. In non-trivial cases, the domain is decomposed
	into patches having separate parameterizations and separate discretization spaces.
	If the discretization spaces agree on the interfaces between
	the patches, the coupling can be done in a conforming way. Otherwise,
	non-conforming discretizations (utilizing discontinuous Galerkin approaches)
	are required. The author and his coworkers have previously introduced
	multigrid solvers for Isogeometric Analysis for the conforming case. In the present
	paper, these results are extended to the non-conforming case. Moreover,
	it is shown that the multigrid solves get even more powerful if the proposed smoother
	is combined with a (standard) Gauss-Seidel smoother.
\end{abstract}


\end{frontmatter}

\section{Introduction}

Isogeometric Analysis (IgA), see~\cite{Hughes:2005}, is a spline based approach
for approximating the solution of a boundary value problem (BVP).
One of the big
strengths of IgA is that it has the approximation power of a high-order method
while the number of degrees of freedom behaves basically like
for a low-order method. To obtain this behavior, we have to be able to increase the
spline degree while we simultaneously increase the smoothness. In IgA, this is
typically called $k$-refinement and leads to spline based discretizations.

In IgA, the computational domain (usually called physical domain) is parameterized
using spline or NURBS functions. Since
it might be too restrictive to parameterize the whole computational
domain using just one global (smooth) geometry function, one typically represents the
physical domain as the union of subdomains, in IgA called
patches. Then, each of the patches is parameterized by its own geometry
function (multi-patch IgA).

On each patch, a space of trial and test functions is introduced.
The simplest approach to set up such a function space is to
use tensor-product B-splines on the unit square or unit cube (parameter domain) and to
use the geometry function to map them onto the physical domain or, in the multi-patch case,
onto one of the patches. If we set up
the function spaces such that the basis functions on the interfaces between
the patches agree, we can use conforming discretizations. Approximation errors
(cf.~\cite{Hughes:2005,Bazilevs:BeiraoDaVeiga:Cottrell:Hughes:Sangalli:2006,
BeiraoDaVeiga:Buffa:Rivas:Sangalli:2011,BeiraoDaVeiga:Buffa:Sangalli:Vazquez:2014,
Takacs:Takacs:2015,Floater:Sande:2017,Sande:Manni:Speleers:2018,deLaRiva:Rodrigo:Gaspar:2018,Tielen:Moller:Goddeke:Vuik:2019} and many others)
and multigrid solvers 
(cf.~\cite{Gahalaut:Kraus:Tomar:2013,Speleers:2015,Hofreither:Takacs:Zulehner:2017,Hofreither:Takacs:2017} and others) for such discretizations have been previously discussed.
Since we are interested in $k$-refinement, we need
results that are explicit in the spline degree. For the single-patch case,
such error estimates have originally been given in~\cite{Takacs:Takacs:2015} and later
improved in~\cite{Sande:Manni:Speleers:2018}. In~\cite{Hofreither:Takacs:2017}, a robust single-patch
multigrid solver has been proposed and analyzed based on the error
estimates from~\cite{Takacs:Takacs:2015}. In~\cite{Takacs:2018}, both the approximation
error estimates and the multigrid solver have been extended to the conforming multi-patch case.
These results are the foundation of the present paper.

If conforming discretizations are not feasible, discontinuous Galerkin
(dG) approaches are possible. One standard dG approach is the Symmetric Interior Penalty
discontinuous Galerkin (SIPG) method, see~\cite{Arnold:1982,Arnold:Brezzi:Cockburn:Marini:2002}.
Already in~\cite{LMMT:2015,LT:2015}, it has been proposed to utilize these approaches
to couple patches in IgA. Recently, also the dependence of
the approximation error on the spline degree has been analyzed, see~\cite{Takacs:2019}.
It was not possible to show that the approximation error is robust in the
spline degree but it could be proven that it only grows logarithmically. 

(Robust) multigrid solvers for such non-conforming discretizations are not known
so far. In the present paper, it is shown
how the multigrid solver from~\cite{Takacs:2018} can be extended to SIPG
discretizations; we observe -- as in~\cite{Takacs:2018} -- that the numerical
experiments show both robustness in the grid size and the spline degree. For
completeness, we also show how to extend the convergence analysis from~\cite{Takacs:2018}
to SIPG discretizations.
It is worth noting that there are alternative solvers for multi-patch IgA, like FETI-type
approaches, cf.~\cite{KleissEtAl:2012,Hofer:2017} and others, overlapping Schwarz type
methods, cf.~\cite{BeiraoDaVeiga:Cho:Pavarino:Scacchi:2012}, or BDDC methods,
cf.~\cite{BeiraoDaVeiga:Cho:Pavarino:Scacchi:2013}; most of them, however, have not been
worked out for the non-conforming case.

Note that the idea behind the proposed subspace corrected mass smoother is that the boundary
value problem on the physical domain (on one patch) can be well approximated by a boundary
value problem on the parameter domain. Thus, the tensor-product structure on the
parameter domain can be used. This is true if the geometry function is not too distorted.
Otherwise, the convergence behavior suffers significantly. The same behavior can be
observed by other fast solvers that are based on the same idea, cf. the fast
diagonalization method~\cite{Sangalli:2016}. Here, the authors have improved their
method by incorporating the geometry information into the preconditioner,
cf.~\cite{Montardini:Sangalli:Tani:2018}. For the multigrid setting,
it has turned out that one can overcome these problems
quite well if the subspace corrected mass smother is combined with a Gauss-Seidel smoother
(hybrid smoother) since both approaches have strengths that seem to be somewhat
orthogonal to each other (robustness in spline degree vs. robustness in the geometry),
cf. also~\cite{Sogn:Takacs:2018}.

In the present paper, we illustrate our findings with numerical
experiments. All presented numerical
experiments are available in the G+Smo library~\cite{gismoweb}.

This paper is organized as follows. We give the model problem
and a conforming discretization in Section~\ref{sec:mp}.
Then, in Section~\ref{sec:sipg}, we discuss why a non-conforming discretization
might be of interest. Moreover, we propose a discontinuous Galerkin approach
that fits our needs. We proceed to multigrid solvers: In Section~\ref{sec:mg:gs},
we discuss Gauss-Seidel smoothers and their performance. Motivated by that
section, we introduce a subspace corrected mass smoother in Section~\ref{sec:mg:scms}
and finally a hybrid smoother in Section~\ref{sec:mg:hyb}. In Section~\ref{sec:fin},
we conclude and give some outlook. The Appending finally contains the
proofs of the theorems stated in the paper.

\section{Model problem and standard Galerkin discretization}\label{sec:mp}

Let $\Omega\subset\mathbb{R}^d$ with $d\in\{2,3\}$ be an open and simply
connected Lipschitz domain. Most of the numerical experiments are done
for the two-dimensional domains shown
in Figure~\ref{fig:1}.

\begin{figure}[h]
\begin{center}
	\includegraphics[height=.25\textwidth]{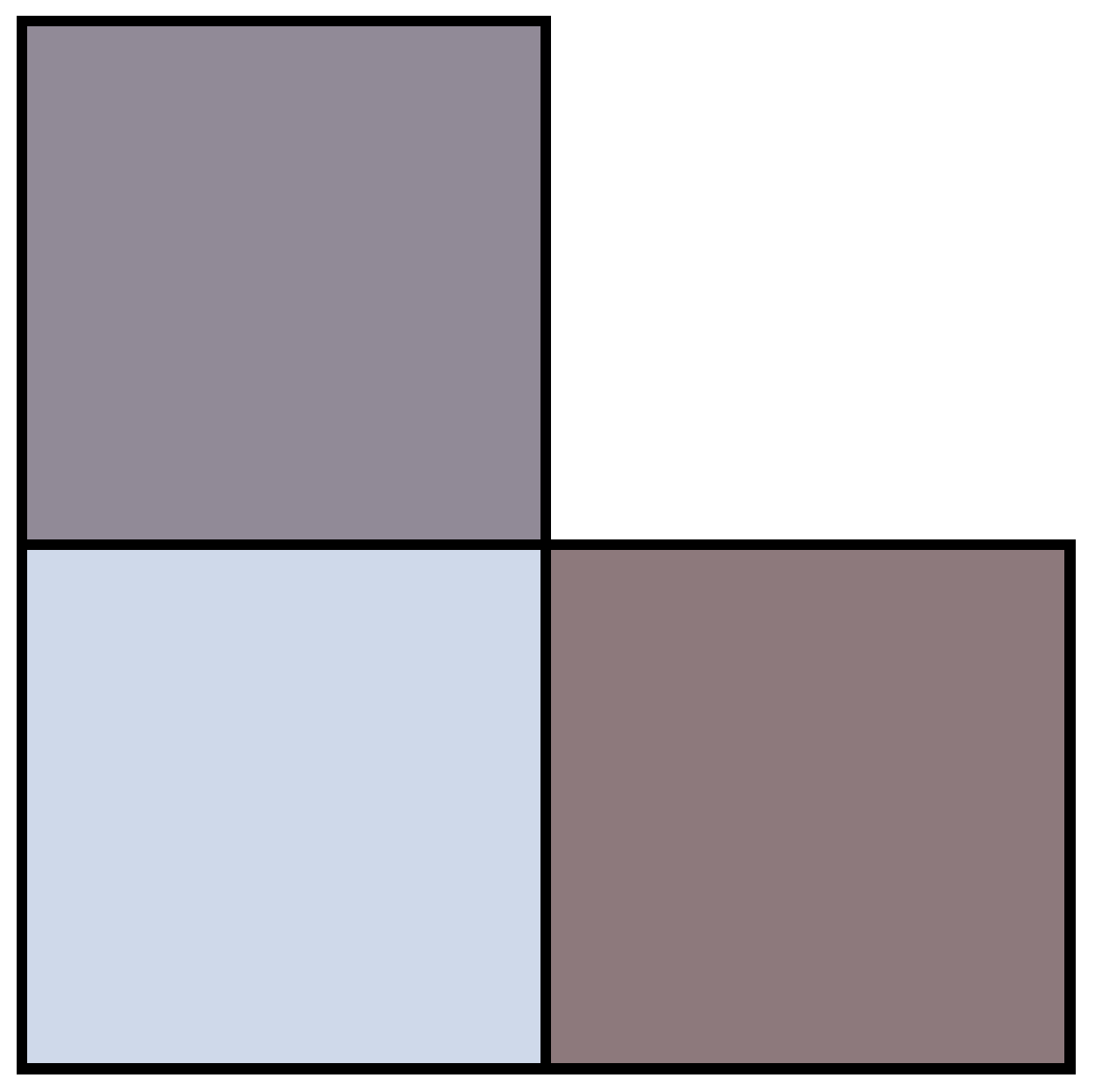} \qquad\qquad
	\includegraphics[height=.25\textwidth]{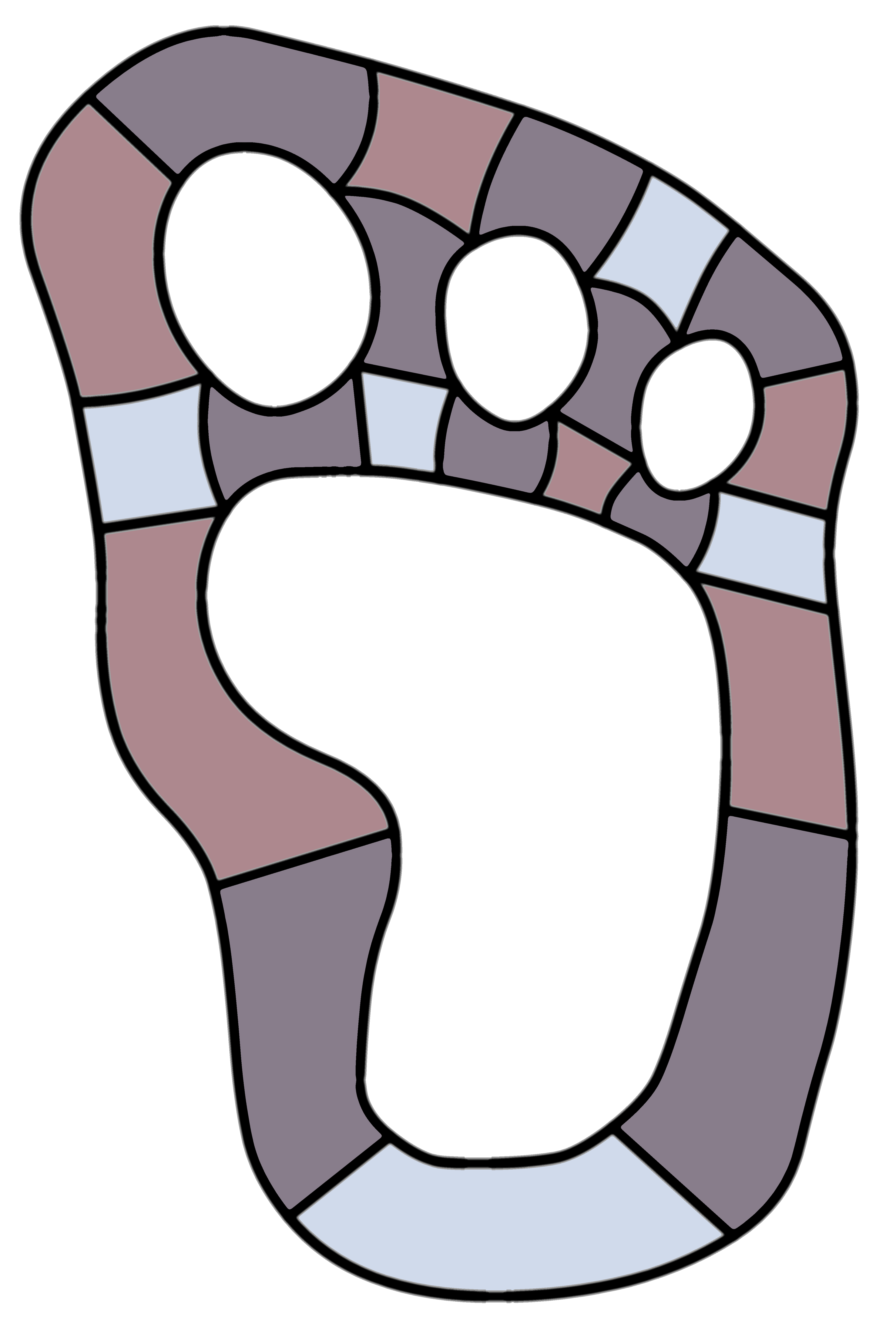}
\end{center}
\caption{The computational domains: L-shaped domain and Yeti footprint}
\label{fig:1}
\end{figure}

The first domain is an L-shaped domain consisting of three quadrilaterals.
Here, the geometry function is just the identity or a translation. 
On the coarsest grid level $\ell=0$,
each patch consists only of one element, i.e., the local basis functions are polynomials only.
The second domain is the Yeti footprint which consists of the 21 patches depicted in
Figure~\ref{fig:1}. Here, the grid on the coarsest grid level is as follows. The five patches
at the bottom consist of two elements each, which are constructed by subdividing the patches
on their longer sides. The remaining patches consist only of one element each.
The grid levels $\ell=1,2,\ldots,$ are obtained by uniform refinement.

Consider the following \emph{Poisson model problem with Dirichlet boundary conditions}.
Find $\mathrm u\in H^1(\Omega)$ such that
\begin{equation} \label{eq:model0}
	-\Delta \mathrm u = \mathrm f \quad \mbox{in}\quad \Omega,
	\qquad
	\mathrm u= \mathrm g \quad \mbox{on}\quad \partial \Omega,
\end{equation}
where $\mathrm f\in L_2(\Omega)$ and $\mathrm g\in H^2(\Omega)$ are given functions.
Here and in what follows,  $L_2(\Omega)$, $H^r(\Omega)$ and $H^r_0(\Omega)$
are the standard Lebesgue and Sobolev spaces.

The experiments are performed for the choice
$\mathrm g(x,y) := \sin(\pi x) \,  \sin(\pi y)$ and $\mathrm f:=-\Delta \mathrm g$;
note that $\mathrm g$ is the exact solution of the problem. 

After homogenization ($u:=\mathrm u-\mathrm g$, $f:=\mathrm f+\Delta\mathrm g$),
the problem reads in variational form
as follows. Find $u\in V:=H^1_0(\Omega)$ such that
\begin{equation} \label{eq:model}
		( \nabla u,\nabla v)_{L_2(\Omega)} = (f,v)_{L_2(\Omega)}
			\qquad \mbox{for all $v \in H^1_0(\Omega)$}.
\end{equation}

The computational domain $\Omega$ is a standard multi-patch domain. Thus, 
we assume that $\Omega$ is composed of $K$ non-overlapping patches $\Omega_k$:
\begin{equation}\label{eq:omega:1}
	\ol\Omega = \bigcup_{k=1}^K \ol{\Omega_k}
	\quad \mbox{ with } \quad
	\Omega_k\cap \Omega_l =\emptyset
	\quad \mbox{ for } k\not=l,
\end{equation}
where each patch is 
represented by a sufficiently smooth bijective geometry function
\begin{equation}\label{eq:omega:2}
		G_k :\widehat{\Omega}:=(0,1)^d \rightarrow \Omega_k := G_k (\widehat{\Omega})\subset \mathbb{R}^d
\end{equation}
which can be continuously extended to $\ol{\widehat{\Omega}}$, the closure of~$\widehat{\Omega}$.
Moreover, we assume that the mesh introduced by the patches satisfies the following
condition.
\begin{assumption}\label{ass:1}
	For any $k\not=l$, the intersection $\ol{\Omega_k} \cap \ol{\Omega_l}$ is either
	(a) empty, (b) one common vertex, (c) the closure of one common edge,
	or -- for $d=3$ -- (d) the closure of one common face.
\end{assumption}

For each of the patches, we assume to have a hierarchy of grids with levels $\ell=0,1,\ldots,L$
obtained by uniform refinement, which we denote by
\begin{equation}\label{eq:local}
	V_{k,\ell} := \{
		v \in L_2(\Omega_k)
		\; : \;
		v \circ G_k \in \bigotimes_{\delta=1}^d S_{p,h_\ell} 
	\} = \mbox{span } \{\varphi_{k,\ell}^{(i)}\}_{i=1}^{N_{k,\ell}},
\end{equation}
where $\bigotimes_{\delta=1}^d S_{p,h_\ell} $ is the space of tensor-product splines of degree
$p$, smoothness $H^p(\widehat{\Omega})$ (or, equivalently, $C^{p-1}(\widehat{\Omega})$)
and grid size $h_\ell=2^{\ell} h_0$ on the parameter domain $\widehat{\Omega}$.
Note that the grid can be non-uniform and both the spline degree and the grid size
can depend on the spatial direction and of the patch number; for simplicity, we do not write
down this dependence explicitly. Note that the grid needs to be quasi-uniform, i.e.,
there needs to be a constant $c>0$ such that all knot spans on grid level $\ell$ are bounded from below
by $c\,h_\ell$. The functions $\varphi_{k,\ell}^{(i)}$ are assumed to form
a (standard) B-spline or NURBS basis of $V_{k,\ell}$.

To be able to set up a conforming discretization, we need to assume
that the function spaces are fully matching on the interfaces, cf.~\cite[Assumption~2.4]{Takacs:2018}.
For tensor-product B-spline bases,
the following assumption characterizes fully matching discretizations.
\begin{assumption}\label{ass:fully}
	On each interface between two patches, the geometry functions,
	the knot vector in tangential direction, and the spline degree in tangential
	direction agree.
\end{assumption}

Assuming a fully matching discretization, we define the conforming discretization space by
\begin{equation}\label{eq:igaspace0}
	V_\ell^{c} :=
	\{
		v \in V
		\; : \;
		v|_{\Omega_k} \in V_{k,\ell} \;\mbox{ for }\; k=1,\ldots,K
	\}.
\end{equation}
A basis for this space is visualized in Figure~\ref{fig:dof} (left), where all
basis functions are represented by their Greville point. The support of the basis functions with Greville
point in the interior of a patch is completely contained in that patch. The basis
functions with Greville points on the interfaces are combinations of the matching
patch-local basis functions. Their support extends to the vertices if and only if
the Greville point is located on the vertex.

\begin{figure}[h]
\begin{center}
	\includegraphics[height=.22\textwidth]{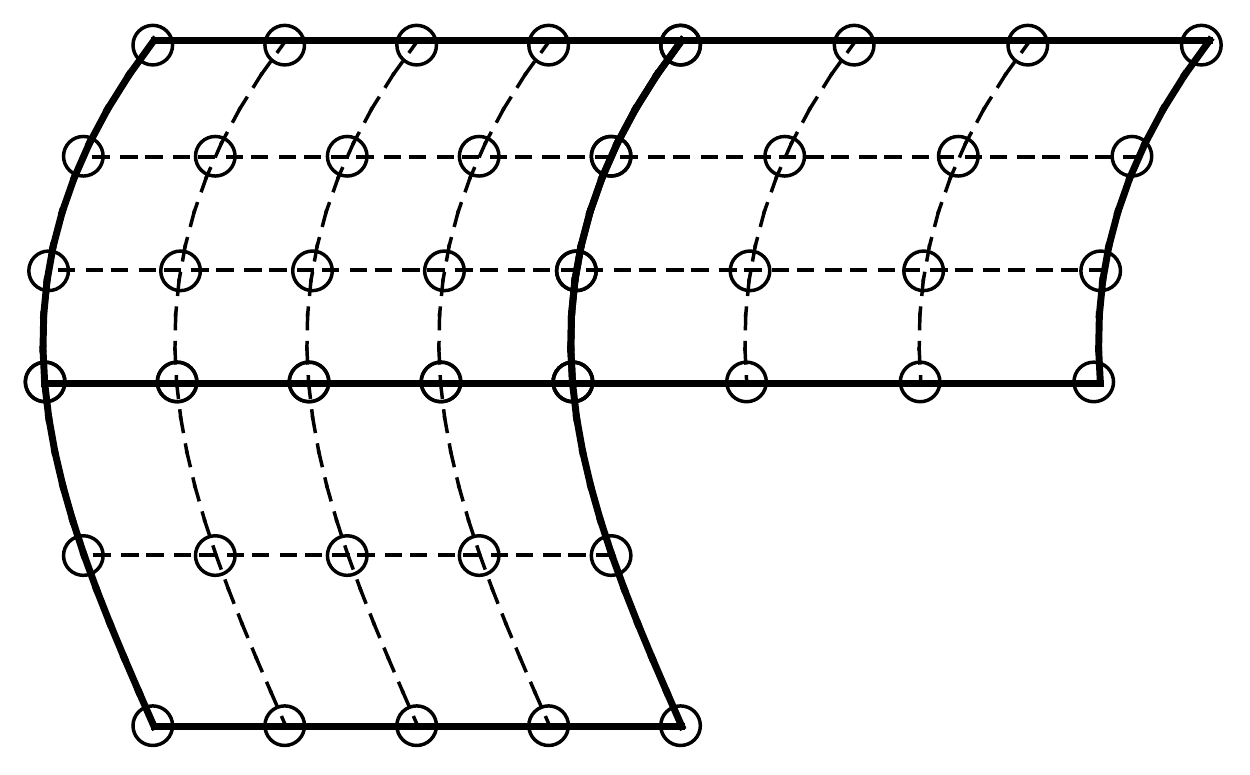}
\end{center}
\caption{Degrees of freedom (represented by Greville points) in conforming case}
\label{fig:dof}
\end{figure}

The conforming discretization of the model problem is obtained using
the standard Galerkin principle:
Find $u_\ell\in V_\ell^{c}$ such that
\begin{equation} \label{eq:model:discr}
		(u_\ell,v_\ell)_{H^1(\Omega)} = (f,v_\ell)_{L_2(\Omega)}
			\qquad \mbox{for all $v_\ell \in V_\ell^{c}$}.
\end{equation}
Using the abovementioned basis for the space $V_\ell^c$, we obtain a standard
matrix-vector problem:
Find $\ul{u}_\ell \in \mathbb{R}^{N_\ell}$ such that
\begin{equation} \label{eq:linear:system:0}
			A_\ell \, \ul{u}_\ell = \ul{f}_\ell,
\end{equation}
where $A_\ell$ is the stiffness matrix, the vector $\ul{u}_\ell$ is the representation of
$u_\ell$ with respect to the chosen basis and the load vector $\ul{f}_\ell$ is obtained
by testing the function $f$ with the basis functions.

\section{Symmetric interior penalty discontinuous Galerkin (SIPG) discretization}\label{sec:sipg}

Following \cite{LMMT:2015,LT:2015,Takacs:2019}, we use a conforming isogeometric discretization for
each patch and couple the patches using discontinuous Galerkin. We assume
that the domain $\Omega$ is again subdivided into patches
such that~\eqref{eq:omega:1}, \eqref{eq:omega:2} and
Assumption~\ref{ass:1} are satisfied. We assume again to have patch-local
spaces $V_{k,\ell}$ as in~\eqref{eq:local}, which are combined in a non-conforming
(i.e., discontinuous) way, i.e., we just define
\begin{equation}\label{eq:igaspace}
	V_\ell^{n} :=
	\{
	v \in L_2(\Omega)
		\; : \;
		v|_{\Omega_k}  \in V_{k,\ell} 
			\; \mbox{ for } \; k=1,\ldots,K \; \mbox{ and } \; v|_{\partial\Omega}=0
	\}.
\end{equation}
This allows us to drop Assumption~\ref{ass:fully}. Note that we strongly enforce the
Dirichlet boundary conditions in our example; alternatively, one could use the
SIPG method also to enforce the Dirichlet boundary conditions.

\begin{figure}[h]
\begin{center}
	\includegraphics[height=.22\textwidth]{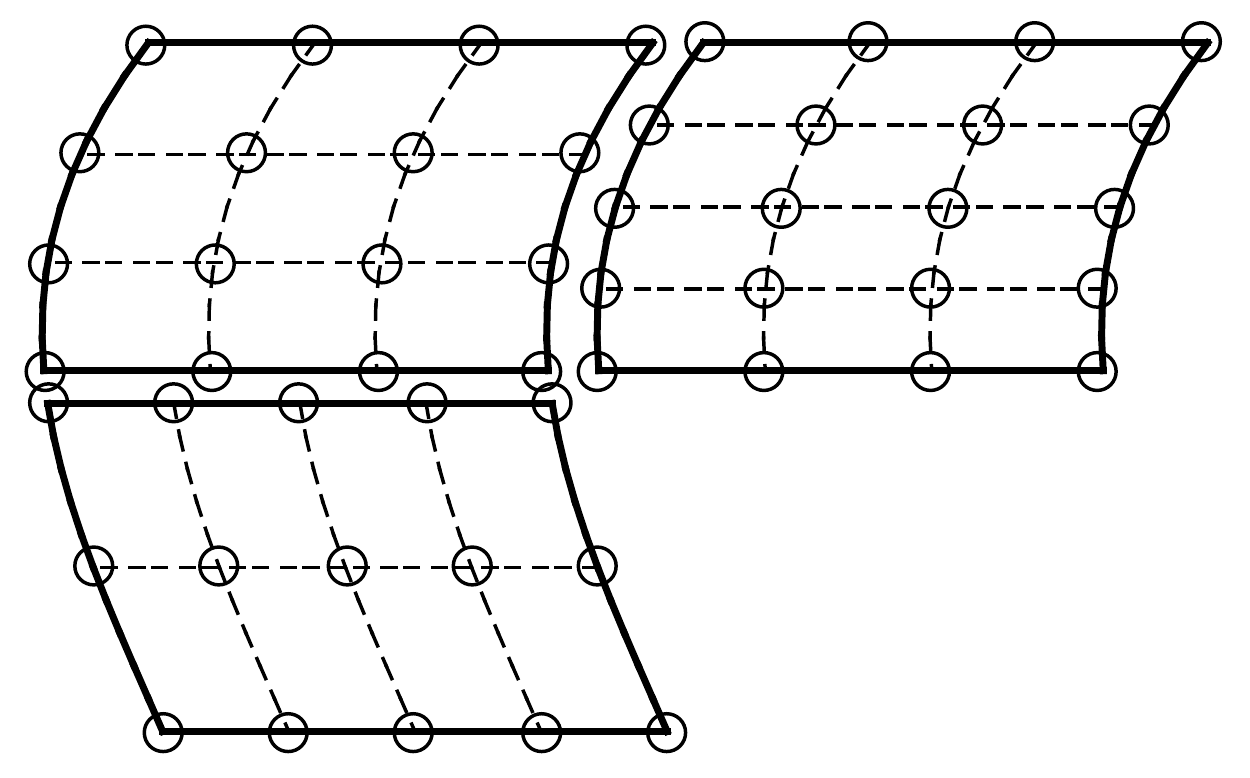}
\end{center}
\caption{Degrees of freedom (represented by Greville points) in non-conforming case}
\label{fig:dof:n}
\end{figure}

Since we have a discontinuous function space, we can visualize the degrees of freedom 
by tearing apart the individual patches, cf. Figure~\ref{fig:dof:n}.
Here, neither the Greville points nor the
basis functions need to agree on the interfaces; the support of each basis function
is contained in one single patch.

Since~$V_\ell^{n}\not\subset V$,
it is not feasible to use the standard Galerkin principle for discretization.
Thus, we couple the patches using the Symmetric Interior
Penalty discontinuous Galerkin (SIPG) method.
First, we define
\[\mathcal{N} := \{ (k,j) \; :\; k<j \mbox{ with } \Omega_k \mbox{ and }
\Omega_l \mbox{ have a common edge.}\}\] to be the set of interface-indices.
For each interface $I_{k,j}$ with $(k,l)\in \mathcal{N}$,
we define the following symbols.
\begin{itemize}
	\item $\nn$ is the outer normal vector of $\Omega_k$. (Thus, 
		$-\nn$ is the outer normal vector of $\Omega_l$.)
	\item $\llbracket \cdot \rrbracket$ is the jump operator:
			$\llbracket u \rrbracket := u|_{\Omega_k} - u|_{\Omega_l}$.
	\item  $\{ \cdot \}$ is the averaging operator:
			$\{ u \} := \tfrac12 ( u|_{\Omega_k} + u|_{\Omega_l})$.
\end{itemize}
Now, we can formulate the SIPG discretization as follows:
Find $u_\ell\in V_\ell^n$ such that
\begin{equation} \label{eq:model:discr}
		(u_\ell,v_\ell)_{A_\ell} = (f,v_\ell)_{L_2(\Omega)}
			\qquad \mbox{for all $v_\ell \in V_\ell^n$},
\end{equation}
where we define
\begin{equation} \label{eq:model:discr2}
\begin{aligned}
		(u,v)_{A_\ell} & := (u,v)_{Q_\ell} - (u,v)_{B_\ell} - (v,u)_{B_\ell},
		\; &(u,v)_{Q_\ell}  := (u,v)_{K_\ell} +  \frac{\sigma p^2}{h_L} (u,v)_{J_\ell},\\
		(u,v)_{K_\ell} & := \sum_{k=1}^K  (\nabla u,\nabla v)_{L_2(\Omega_k)},
		 &(u,v)_{J_\ell}  := \sum_{(k,l)\in \mcN}   (\llbracket u\rrbracket,\llbracket v\rrbracket)_{L_2(I_{k,l})} ,	\\
		(u,v)_{B_\ell} & := \sum_{(k,l)\in \mcN} (\llbracket u \rrbracket,\{\nabla v\}\cdot \nn)_{L_2(I_{k,l})}.
\end{aligned}
\end{equation}
There is some $\sigma_0>0$ independent of the grid size, the spline degree and the
number of patches such that for all $\sigma \ge \sigma_0$, the bilinear form
$(\cdot,\cdot)_{A_\ell}$ is coercive, cf.~\cite[Theorems~8 and~9]{Takacs:2019}.
Thus, for $\sigma \ge \sigma_0$, the Theorem of Lax Milgram states
that the problem~\eqref{eq:model:discr} has exactly one solution. The
combination of Ce\'a's Lemma and a naive approximation
error estimate yields a discretization error estimate of the form
\[
		|u-u_L|_{Q_L}^2 \le c\,  p^2 h_L^2 |u|_{H^2(\Omega)}^2,
\]
cf.~\cite{Takacs:2019}. By doing a more careful analysis, we obtain 
estimates of the form
\[
		|u-u_L|_{Q_L}^2 \le c\, (\log p)^4 h_L^2 |u|_{H^2(\Omega)}^2,
\]
see~\cite[eq.~(15)]{Takacs:2019}. This significantly decreases the influence of the spline degree.

Note that the penalization term has the form
\[
		\frac{\sigma p^2}{h_L},
\]
i.e., it depends on the grid size on the finest grid $h_L$. 
This follows the ideas from~\cite{Gopalakrishnan:Kanschat:2003}. The idea behind that
is that 
\begin{equation}\label{eq:a:galerkin}
		(u_\ell,v_\ell)_{A_\ell} = (u_\ell,v_\ell)_{A_{\ell+1}}
		\quad \mbox{and}\quad
		(u_\ell,v_\ell)_{Q_\ell} = (u_\ell,v_\ell)_{Q_{\ell+1}}
\end{equation}
hold, i.e., we obtain a multigrid solver with conforming coarse-grid correction.
This means that -- on the coarse grid levels -- the discretization is over penalized by a factor of $2^{L-\ell}$, i.e.,
\[
 	\underbrace{\frac{\sigma p^2}{h_L}}_{\displaystyle \widetilde\Sigma_\ell:=} = 2^{L-\ell} \underbrace{\frac{\sigma p^2}{h_\ell}}_{\displaystyle \Sigma_\ell:=},
\]
where $\Sigma_\ell$ is the canonical parameter and $\widetilde\Sigma_\ell$ is the chosen
one.
We will see that this does not cause any problems for the examples we consider;
following~\cite{Gopalakrishnan:Kanschat:2003}, convergence theory only holds if the
number of smoothing steps is sufficiently increased for the coarser grid
levels, cf. Remark~\ref{rem:opcosts}.

Using a basis for the space $V_\ell^n$, we obtain a standard matrix-vector problem:
Find $\ul{u}_\ell \in \mathbb{R}^{N_\ell}$ such that
\begin{equation} \label{eq:linear:system}
			A_\ell \ul{u}_\ell = \ul{f}_\ell.
\end{equation}

\section{Multigrid solvers with Gauss-Seidel smoothers}\label{sec:mg:gs}

In this and the following sections, we discuss several possible choices of
multigrid smoothers, illustrate their convergence behavior with
numerical experiments, and comment on the convergence theory.

We consider conforming discretizations and non-conforming discretizations
which are set up as discussed in the last two sections.
As we have nested spaces in all cases, the matrix $I_{\ell-1}^{\ell}$ is always the canonical embedding
from $V_{\ell-1}$ into $V_{\ell}$ and the restriction matrix $I_{\ell}^{\ell-1}$ is 
its transpose.
The chosen method is presented as pseudo-code as Algorithm~\ref{alg:1},
where we choose $\mu=1$ for the V-cycle method or $\mu=2$ for the W-cycle method.

\begin{algorithm}[h]
\textsc{Multigrid}$\left(\ell, \ul{f}_\ell, \ul{u}_\ell\right)$\\
\mbox{}\qquad// Pre-Smoothing\\
\mbox{}\qquad\textbf{for} $n=1,\ldots,\nu_\ell$\\
\mbox{}\qquad\mbox{}\qquad $\ul{u}_\ell\gets \ul{u}_\ell +  L_\ell^{-1}
                                    \left(\ul{f}_\ell-A_\ell\;\ul{u}_\ell\right)$ \\
\mbox{}\qquad// Coarse-grid correction\\
\mbox{}\qquad\textbf{if} $\ell=1$\\
\mbox{}\qquad\mbox{}\qquad $\ul{u}_\ell\gets \ul{u}_\ell+ I_{\ell-1}^{\ell} A_{\ell-1}^{-1} I_{\ell}^{\ell-1}\left(\ul{f}_\ell - A_\ell
                      \;\ul{u}_\ell\right)$\quad// Direct solver\\
\mbox{}\qquad\textbf{else}\\
\mbox{}\qquad\mbox{}\qquad\textbf{for} $n=1,\ldots,\mu$\\
\mbox{}\qquad\mbox{}\qquad\mbox{}\qquad $\ul{u}_\ell\gets \ul{u}_\ell + I_{\ell-1}^{\ell}
			\textsc{Multigrid}\left(\ell-1, 
									I_{\ell}^{\ell-1}\left(\ul{f}_\ell - A_\ell
                      \;\ul{u}_\ell\right), 0\right)$
                      \\
\mbox{}\qquad// Post-Smoothing\\
\mbox{}\qquad\textbf{for} $n=1,\ldots,\nu_\ell$\\
\mbox{}\qquad\mbox{}\qquad $\ul{u}_\ell\gets \ul{u}_\ell +  L_\ell^{-\top}
                                    \left(\ul{f}_\ell-A_\ell\;\ul{u}_\ell\right)$ \\
\mbox{}\qquad\textbf{return} $\ul{u}_\ell$
\caption{\label{alg:1}Multigrid algorithm}
\end{algorithm}

In the finite element world, Gauss-Seidel smoothers are known to be
very efficient smoothers; thus, as a first attempt, we consider such a smoother.
One forward Gauss-Seidel sweep can be represented by
\[
		\ul{u}_\ell \gets \ul{u}_\ell +  L_\ell^{-1}
                                    \left(\ul{f}_\ell-A_\ell\;\ul{u}_\ell\right),
\]
where $L_\ell$ is a lower-triangular matrix containing the coefficients
of the stiffness matrix $A_\ell$, i.e., it is given by
\[
			(L_\ell)_{i,j} = \left\{\begin{array}{ll}
						(A_\ell)_{i,j} & \mbox{ if } i \ge j \\
						0 & \mbox{ if } i < j \\
				\end{array}\right..
\]
To be able to use our multigrid solver as preconditioner for a 
conjugate gradient solver, the post-smoothing procedure uses
the transposed matrix $L_\ell^\top$, which represents a backward-Gauss-Seidel
sweep.

All tables show the number of iterations
required until the stopping criterion
\[
		\frac{\|A_L \ul{u}_L - \ul{f}_L\|_{\ell^2}}{\|\ul{f}_L\|_{\ell^2}} \le \epsilon := 10^{-8}	
\]
is satisfied.

As usual, the convergence behavior of the overall solver
can be improved if the multigrid method is not just used directly as a solver,
but as a preconditioner within a preconditioned conjugate gradient (PCG) solver.
Thus, we present results for both possibilities; in the following sections we will restrict
ourselves to the more efficient PCG solver. Since the V-cycle and the
W-cycle methods yield comparable iteration counts, we present the results for the
more efficient V-cycle only. The number of smoothing steps is chosen as $\nu_\ell:=1$ in
all cases.

The multigrid solver was implemented in C++ based on the G+Smo
library~\cite{gismoweb}. 
The tables shown in the remainder of this section are obtained with the following
command line instructions, where the values $L$ and $p$ are substituted accordingly.

\begin{center}
\begin{minipage}{.95\textwidth}
\begin{lstlisting}[mathescape,columns=flexible,basicstyle=\ttfamily]
> git clone https://github.com/gismo/gismo.git
> cd gismo
> make
> cd build/bin
> ./multiGrid_example -g domain2d/ldomain.xml -r $L$ -p $p$
      -s gs -i d                                         $\tabnr{tab:GS1}{a}$
> ./multiGrid_example -g domain2d/yeti_mp2.xml -r $L$ -p $p$
      -s gs -i d                                         $\tabnr{tab:GS2}{a}$
\end{lstlisting}
\end{minipage}
\end{center}
The results
for the PCG experiments, presented in Tables~\ref{tab:GS1}~(b) and~\ref{tab:GS2}~(b),
are obtained by replacing the option {\tt -i d} by the option {\tt -i cg}.

\begin{table}[ht]
\begin{center}
    \begin{tabular}{l|rrrrrrr|rrrrrrr}
    \toprule
						& \multicolumn{7}{c|}{(a) Direct -- Conforming}  & \multicolumn{7}{c}{(b) PCG -- Conforming}  \\    
    \midrule
    $L\,\backslash\, p$    &  2 &  3 &  4 &  5 &  6 &  7 &  8    &  2 &  3 &  4 &  5 &  6 &  7 &  8\\
    \midrule
	 4                        &  9 & 24 & 74 &\dv &\dv &\dv &\dv    &  8 & 15 & 28 & 53 &\dv &\dv &\dv\\
	 5                        &  9 & 24 & 73 &\dv &\dv &\dv &\dv    &  8 & 15 & 28 & 52 &\dv &\dv &\dv\\
	 6                        &  9 & 24 & 72 &\dv &\dv &\dv &\dv    &  8 & 15 & 28 & 53 &\dv &\dv &\dv\\
	 7                        & 10 & 24 & 72 &\dv &\dv &\dv &\dv    &  8 & 15 & 28 & 54 &\dv &\dv &\dv\\
	 8                        & 10 & 24 & 72 &\dv &\dv &\dv &\dv    &  8 & 15 & 28 & 54 &\dv &\dv &\dv\\
	\bottomrule
	\end{tabular}
\end{center}
\caption{V-cycle with Gauss-Seidel smoother for the L-shaped domain}
\label{tab:GS1}
\end{table}
\begin{table}[ht]
\begin{center}
    \begin{tabular}{l|rrrrrrr|rrrrrrr}
    \toprule
						& \multicolumn{7}{c|}{(a) Direct -- Conforming}  & \multicolumn{7}{c}{(b) PCG -- Conforming}  \\    
    \midrule
    $L\,\backslash\, p$    &  2 &  3 &  4 &  5 &  6 &  7 &  8    &  2 &  3 &  4 &  5 &  6 &  7 &  8\\
    \midrule
	 3                        & 13 & 25 & 75 &\dv &\dv &\dv &\dv    & 10 & 15 & 28 & 54 &\dv &\dv &\dv\\
	 4                        & 14 & 25 & 74 &\dv &\dv &\dv &\dv    & 10 & 15 & 28 & 53 &\dv &\dv &\dv\\
	 5                        & 15 & 25 & 74 &\dv &\dv &\dv &\dv    & 11 & 16 & 28 & 54 &\dv &\dv &\dv\\
	 6                        & 15 & 25 & 72 &\dv &\dv &\dv &\dv    & 11 & 16 & 29 & 54 &\dv &\dv &\dv\\
	 7                        & 17 & 25 & 73 &\dv &\dv &\dv &\dv    & 12 & 16 & 29 & 55 &\dv &\dv &\dv\\
	\bottomrule
	\end{tabular}
\end{center}
\caption{V-cycle with Gauss-Seidel smoother for the Yeti footprint}
\label{tab:GS2}
\end{table}

In Table~\ref{tab:GS1}, we observe that the multigrid solver is certainly
robust in the grid size. While this approach is very efficient
for small numbers of spline degrees, we observe that the convergence rates deteriorate
significantly if the spline degree is increased. On the right side of the table, one
can see the iteration counts for a preconditioned conjugate gradient method where one
V-cycle of the mentioned multigrid method is used as a preconditioner. We observe that
the iteration counts are significantly smaller than the iteration counts obtained by
directly applying the multigrid solver. However, we simultaneously observe that we do
not observe any qualitative improvement.
In Table~\ref{tab:GS2}, we give the iteration counts for the Yeti footprint. We observe that
-- despite the fact that the geometry function is now non-trivial --
the iteration counts are very similar to those of the L-shaped domain.

When one turns to the non-conforming discretizations, it immediately turns out that the
multigrid solver utilizing the Gauss-Seidel smoother does not converge well at all.

One can show using standard arguments that the multigrid method converges with rates
that are independent of the grid size and of the number of patches.
The convergence analysis (for the conforming case) employs estimates
that increase exponentially in the spline degree, cf.~\cite{Gahalaut:Kraus:Tomar:2013}.
The numerical experiments show that this is not only a matter of the proof.

Since we did not obtain convincing results,
we are interested in more advanced smoothers that work well also for SIPG discretizations
and which do not deteriorate if the spline degree is increased.

\section{Multigrid with subspace corrected mass smoother}\label{sec:mg:scms}

In this section, we employ the subspace corrected mass smoother as introduced
in~\cite{Hofreither:Takacs:2017}. That smoother requires that the spline space
has a tensor-product structure; in our examples, we have such a structure on each
patch but not on the overall domain.
The extension of that smoother to conforming discretizations has been discussed
in~\cite{Takacs:2018} based on a domain-decomposition approach. The key idea was to
decompose all degrees of freedom on a per-piece bases. Pieces are the patch-interiors and the
interface pieces. In two dimensions, the interface pieces are the edges and the
vertices of each of the patches. In three dimensions, the interface pieces are the faces,
the edges and the vertices of each of the patches. The decomposition of the degrees of
freedom is depicted in Figure~\ref{fig:decomp} (left).
\begin{figure}[h]
\begin{center}
	\includegraphics[height=.22\textwidth]{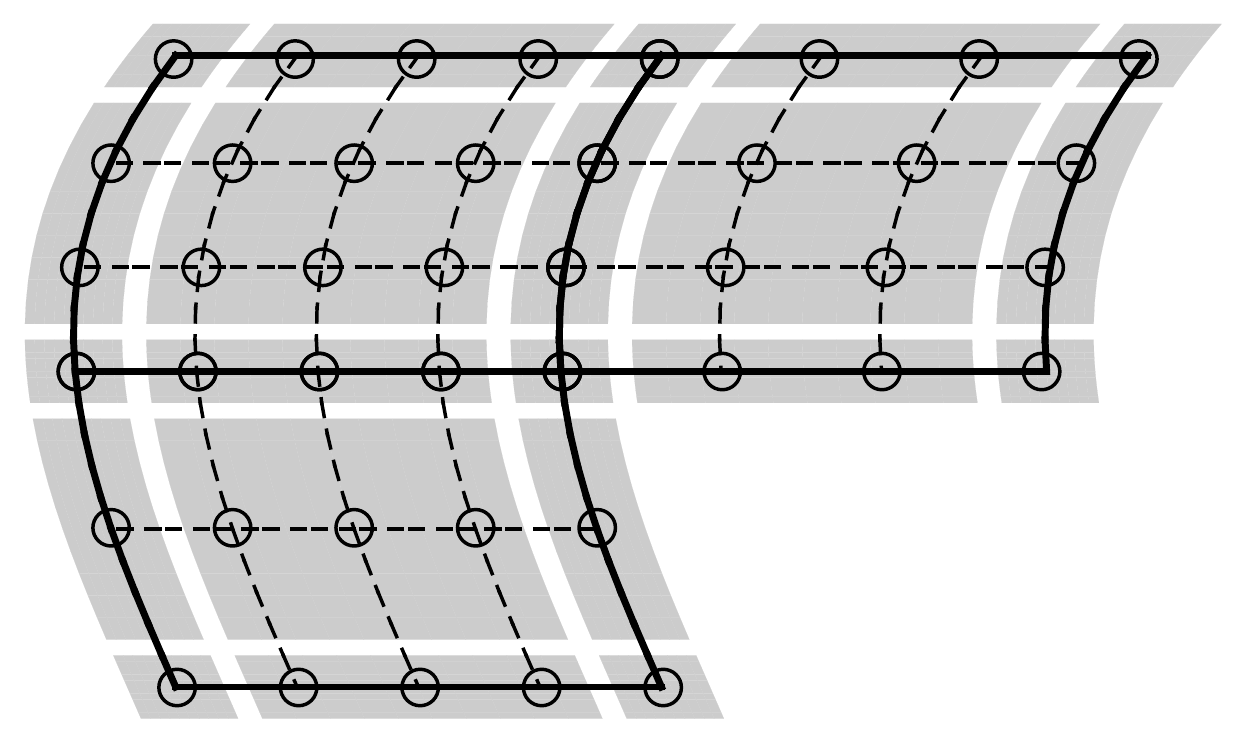}
	\includegraphics[height=.22\textwidth]{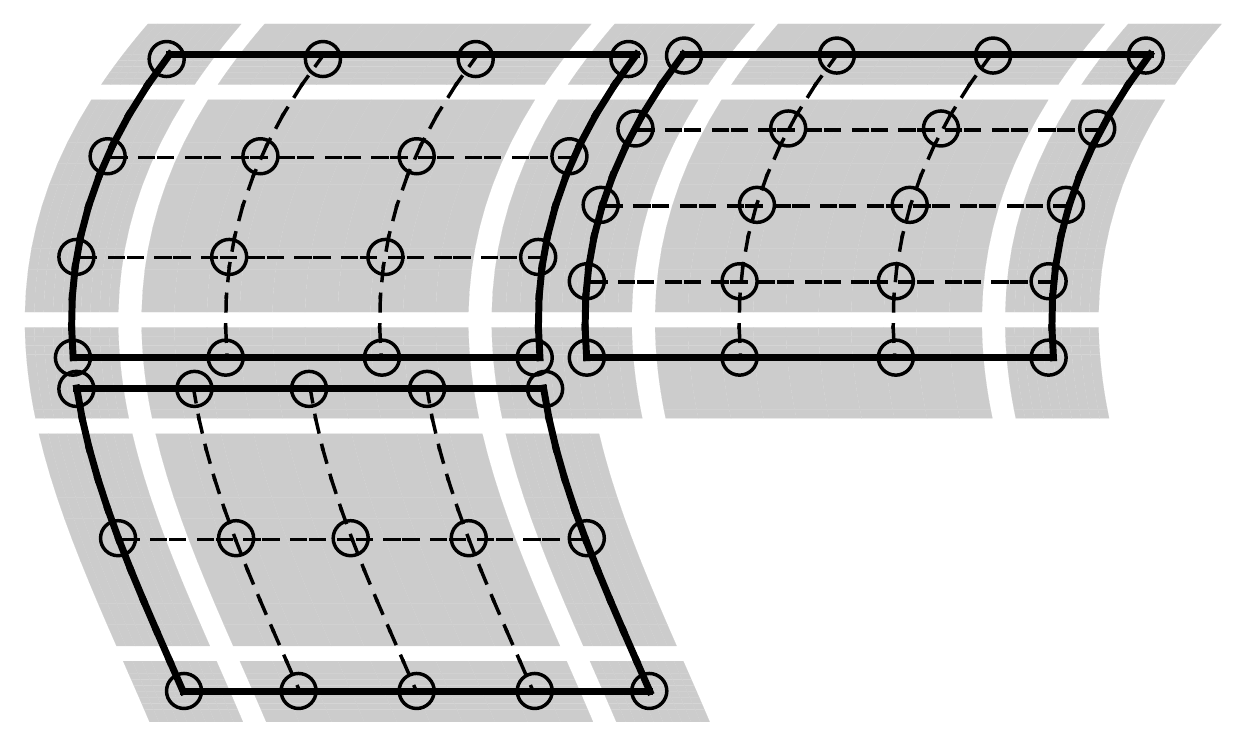}
\end{center}
\caption{Decomposition of the degrees of freedom (represented by the Greville point)}
\label{fig:decomp}
\end{figure}

The piece-local smoothers, which we denote by $L_{\ell,T}^{-1}$, are defined
as follows. For $T$ being a patch-interior, we choose
$L_{\ell,T}^{-1}$ to be the subspace corrected mass smoother as proposed
in~\cite{Hofreither:Takacs:2017}. We choose the scaling parameter
(which was called~$\sigma$ in~\cite{Hofreither:Takacs:2017}) to be
$\delta^{-1} h_\ell^{-2}$ for some suitable chosen
parameter $\delta>0$. If $T$ is an interface piece, we choose
\[
       L_{\ell,T}:= P_T^{\top} A_\ell P_T,
\]
where the matrix $P_{\ell,T}$ represents the embedding of the piece $T$ in the whole space.
The symbol $L_{\ell,T}^{-1}$ refers to be the application of a direct solver. 
Applying a direct solver on the interfaces is
feasible since the interfaces have much smaller numbers of degrees of freedom
than the interiors of the patches.
The overall smoother is just an additive composition of the piece-local smoothers, i.e., we
choose
\begin{equation}\label{eq:Ldef}
		L_\ell^{-1} := \tau \sum_T P_{\ell,T} L_{\ell,T}^{-1} P_{\ell,T}^\top,
\end{equation}
where the sum is taken over all pieces $T$. Here, $\tau>0$
is some damping parameter to be chosen.

The convergence theory from~\cite{Takacs:2018} can be summarized by the following theorem.

\begin{theorem}\label{thrm:converg:conf}
	Assume that $\Omega \subset \mathbb{R}^2$
	is such that full elliptic regularity is satisfied
	(cf.~\cite[Assumption~3.1]{Takacs:2018}).
	Consider the conforming discretization and a multigrid solver
	with the smoother~\eqref{eq:Ldef}. 
	There are constants $\tau^*$, $\delta^*$ and $\theta$ which are independent of $K$, $h$, $L$ and $p$
	(but may depend particularly on the geometry functions
	and the maximum number of neighbors of a patch) such that for
	\begin{equation}\label{eq:nu:1}
		\tau \in (0,\tau^*),\qquad \delta \in (0,\delta^*) \qquad\mbox{and}\qquad
		\nu_\ell > \nu_\ell^* := p\; \frac{\tau^*}{\tau} \; \frac{\delta^*}{\delta}\; \theta,
	\end{equation}
	the W-cycle multigrid method converges with
	a convergence rate $q\le \max_\ell \nu_\ell^*/\nu_\ell$.
\end{theorem}

Note that the terms $\frac{\tau^*}{\tau}$ and $\frac{\delta^*}{\delta}$ imply that the
convergence degrades if too small values of those parameters are chosen. Thus, it is
of interest to choose these parameters in a rather optimal way.

Note that the convergence theorem requires full elliptic regularity
(cf.~\cite[Assumption~3.1]{Takacs:2018}). Thus, it is applicable to the
Yeti footprint but it is not directly applicable to the L-shaped domain. 
Convergence results for the case with
full elliptic regularity often carry over \emph{in practice} to cases where
that regularity assumption does not hold. The same behavior can be observed
for the numerical experiments we have considered.

Observe that the convergence theorem suggests that the number of smoothing steps
should increase with $p$. As already outlined in~\cite{Takacs:2018}, this seems
to be too pessimistic since the numerical experiments have shown that $\nu_\ell:=1$
in all cases yields good convergence rates.

The numerical experiments are again applied within the same setup as in the
last section. We set up a V-cycle multigrid method with 1+1 smoothing steps
of the proposed smoother $L_\ell$ (on all grid levels). The damping parameter~$\tau$
is chosen as indicated with the option {\tt --MG.Damping} and the scaling parameter~$\delta$
is chosen as indicated with the option {\tt --MG.Scaling} below.
The tables for the conforming case shown in this section are obtained with the following
code, where the values $L$ and $p$ are substituted accordingly:
\begin{center}
\begin{minipage}{.95\textwidth}
\begin{lstlisting}[mathescape,columns=flexible,basicstyle=\ttfamily]
> ./multiGrid_example -g domain2d/ldomain.xml -r $L$ -p $p$
      -s scms --MG.Damping 1 --MG.Scaling .12 -i d        $\tabnr{tab:SCMS1}{a}$
> ./multiGrid_example -g domain2d/yeti_mp2.xml -r $L$ -p $p$
      -s scms --MG.Damping .25 --MG.Scaling .2 -i cg     $\tabnr{tab:SCMS3}{a}$
\end{lstlisting}
\end{minipage}
\end{center}
The results for the PCG experiments presented Table~\ref{tab:SCMS1}~(b)
are obtained by replacing the option {\tt -i d} by the option {\tt -i cg}.

\begin{table}[ht]
\begin{center}
    \begin{tabular}{l|rrrrrrr|rrrrrrr}
    \toprule
						& \multicolumn{7}{c|}{(a) Direct -- Conforming}  & \multicolumn{7}{c}{(b) PCG -- Conforming}  \\    
    \midrule
    $L\,\backslash\, p$    &  2 &  3 &  4 &  5 &  6 &  7 &  8    &  2 &  3 &  4 &  5 &  6 &  7 &  8\\
    \midrule
	 4                        & 27 & 23 & 22 & 20 & 17 & 16 & 15    & 16 & 14 & 13 & 12 & 11 & 11 & 10\\
	 5                        & 29 & 27 & 27 & 26 & 24 & 24 & 22    & 17 & 16 & 15 & 15 & 14 & 14 & 14\\
	 6                        & 30 & 30 & 28 & 27 & 27 & 27 & 26    & 17 & 17 & 16 & 16 & 15 & 16 & 15\\
	 7                        & 31 & 30 & 29 & 28 & 28 & 27 & 27    & 17 & 17 & 17 & 16 & 16 & 16 & 16\\
	 8                        & 32 & 31 & 30 & 29 & 28 & 28 & 28    & 18 & 17 & 17 & 17 & 16 & 16 & 16\\	\bottomrule
	\end{tabular}
\end{center}
\caption{V-cycle with subspace corrected mass smoother for the L-shaped domain}
\label{tab:SCMS1}
\end{table}

All numerical experiments show that the proposed method is robust both in the grid
size and the spline degree. However, when comparing the results for the Yeti footprint
from Table~\ref{tab:SCMS3}~(a) with the corresponding results for the L-shaped domain
from Table~\ref{tab:SCMS1}~(b), we see that the multigrid solver suffers from
distorted geometry functions.

The numbers for the Yeti footprint seem not to be
completely robust in the grid size. Since we have given a convergence
theory, we know that the convergence numbers are bounded uniformly. Thus, the observed
behavior is pre-asymptotic. The reason for this is that on coarser grid levels, the
geometry is not resolved exactly. Let $A_\ell$ be the original stiffness matrix and
$\widehat{A}_\ell$ be the simplified stiffness matrix obtained by neglecting the geometry
function. Then, we have
\[
	\kappa( \widehat{A}_\ell^{-1} A_\ell )
			= \sup_{v_\ell \in V_\ell} \frac{|v_\ell|_{H^1(\Omega)}^2 }{\sum_{k=1}^K|v_\ell \circ G_k|_{H^1(\widehat{\Omega})}^2 }
			\sup_{v_\ell \in V_\ell}\frac{\sum_{k=1}^K|v_\ell \circ G_k|_{H^1(\widehat{\Omega})}^2 }{ |v_\ell|_{H^1(\Omega)}^2 },
\]
which yields
\begin{equation}\label{eq:sgn}
\begin{aligned}
	\kappa( \widehat{A}_0^{-1} A_0 ) & \le \cdots \le
	\kappa( \widehat{A}_{L-1}^{-1} A_{L-1}) \le
	\kappa( \widehat{A}_{L}^{-1} A_{L}) \\&\le
	\sup_{v \in V} \frac{ |v|_{H^1(\Omega)}^2 }{\sum_{k=1}^K|v \circ G_k|_{H^1(\widehat{\Omega})}^2 }
			\sup_{v \in V}\frac{\sum_{k=1}^K|v \circ G_k|_{H^1(\widehat{\Omega})}^2 }{ |v|_{H^1(\Omega)}^2 },
\end{aligned}
\end{equation}
which can be finally bounded by a constant times some power of the quantity
$\|\nabla G\|_{L_\infty(\widehat{\Omega})}\|(\nabla G)^{-1}\|_{L_\infty(\widehat{\Omega})}$.
Of none of these relations is satisfied by equality. In such a case, the iteration
counts are likely to increase if the grid gets refined.
For more on this topic, see~\cite[Section~7.4]{Sogn:2018}.

As a next step, we turn towards the non-conforming discretizations. 
Here, each degree of freedom is assigned to exactly one patch. So,
it would be tempting to set up a patch-wise splitting of the degrees of
freedom. Unfortunately, numerical experiments have shown that this approach
does not work well. So, we follow the approach from~\cite{Takacs:2018} also
in the non-conforming case and split the degrees of freedom again into pieces $T$.
This means that we avoid breaking the coupling which was enforced by the penalty term.
So, the degrees of freedom belonging to
one edge (face, vertex) are considered to be one piece, even if the degrees of
freedom belong to different patches, see Figure~\ref{fig:decomp} (right).

For this choice, we can give the following convergence theorem.
\begin{theorem}\label{thrm:converg}
	Assume that $\Omega\subset \mathbb{R}^2$ is such that full elliptic regularity holds
	(cf.~\cite[Assumption~4]{Takacs:2019}) and assume that the geometry
	functions (but not necessarily the discretizations) agree on the
	interfaces (cf.~\cite[Assumption~2]{Takacs:2019}).
	Consider the SIPG discretization and a multigrid solver
	with the smoother~\eqref{eq:Ldef}.
	There are constants $\tau^*$, $\delta^*$ and $\theta$ which are independent of $K$, $h$, $L$ and $p$
	(but may depend particularly on the geometry functions
	and the maximum number of neighbors of a patch) such that for
	\begin{equation}\label{eq:nu:2}
		\tau \in (0,\tau^*), \quad \delta \in (0,\delta^*) \quad\mbox{and}\quad
		\nu_\ell > \nu_\ell^* := 
		2^{L-\ell}\;(1+L-\ell)^2\;p\;(\log p)^4\; \frac{\tau^*}{\tau} \; \frac{\delta^*}{\delta}\; \theta,
	\end{equation}
	the W-cycle multigrid method converges with
	a convergence rate $q\le \max_\ell \nu_\ell^*/\nu_\ell$.
\end{theorem}

We give the proof of this theorem in the Appendix; the proof is based on
the error estimates from~\cite{Takacs:2019}.

One might observe that the number of smoothing steps required by this convergence
theorem increases like $(1+L-\ell)2^{L-\ell}$. This follows the approach suggested
in~\cite{Gopalakrishnan:Kanschat:2003} and is related to the chosen over-penalization
discussed in Section~\ref{sec:sipg}.
\begin{remark}\label{rem:opcosts}
Note that the
number of degrees of freedom on the coarser grid levels is smaller by a factor of
$2^{d(L-\ell)}$. So, also using these additional smoothing steps, the overall complexity of
the multigrid solver is still linear in the number of unknowns on the finest grid level if
(a) $d\ge3$ or (b) the V-cycle is considered. If we consider $d=2$ and the W-cycle, the
choice~\eqref{eq:nu:2} yields that the computational complexity grows like
$N_L L^3$, where $N_L$ is the number of unknowns on the finest grid level.
In the numerical experiments, we did not observe that increasing the number
of smoothing steps has been required. Analogously to the conforming case, also
the stated dependence on $p$ is too pessimistic; thus, we again
choose $\nu_\ell:=1$ on all grid levels.
\end{remark}

Now, we provide numerical experiments for the SIPG discretization.
Theoretically, we could just use exactly the discretization that has been chosen
for the conforming case. This, however, yields a (particularly uninteresting)
special case since Assumption~\ref{ass:fully} holds. In this special case, we have
$V_\ell^c \subset V_\ell^n$ and the SIPG formulation converges to the conforming
discretization for $\sigma \rightarrow \infty$. Instead, we are interested
in a discretization such that Assumption~\ref{ass:fully} does not hold: We
modify the setup of the spaces. For one third of the patches, we use the original
spline space $S_{p,h_\ell}(\widehat{\Omega})$. For one third of the patches, we use
the spline space $S_{p+1,2h_\ell}(\widehat{\Omega})$. For the last third of the
patches, we use the spline space $S_{p,2h_\ell}(\widehat{\Omega})$. This particular
setting is obtained with the command line option {\tt --NonMatching}.
In this way, we obtain a setup where a conforming discretization is not possible.

The tables for the non-conforming case shown in this section are obtained with the following
code, where the values $L$ and $p$ are substituted accordingly:
\begin{center}
\begin{minipage}{.95\textwidth}
\begin{lstlisting}[mathescape,columns=flexible,basicstyle=\ttfamily]
> ./multiGrid_example -g domain2d/ldomain.xml -r $L$ -p $p$ --DG
      --NonMatching -s scms --MG.Damping .9 --MG.Scaling .12
      -i d                                            $\tabnr{tab:SCMS2}{a}$
> ./multiGrid_example -g domain2d/yeti_mp2.xml -r $L$ -p $p$ --DG
      --NonMatching -s scms --MG.Damping .25 --MG.Scaling .2
      -i cg                                           $\tabnr{tab:SCMS3}{b}$
\end{lstlisting}
\end{minipage}
\end{center}
The results for the PCG experiments presented in Table~\ref{tab:SCMS2}~(b)
are obtained by replacing the option {\tt -i d} by the option {\tt -i cg}.

\begin{table}[ht]
\begin{center}
    \begin{tabular}{l|rrrrrrr|rrrrrrr}
    \toprule
						& \multicolumn{7}{c|}{(a) Direct -- Non-conforming}  & \multicolumn{7}{c}{(b) PCG -- Non-conforming}  \\    
    \midrule
    $L\,\backslash\, p$    &  2 &  3 &  4 &  5 &  6 &  7 &  8    &  2 &  3 &  4 &  5 &  6 &  7 &  8\\
    \midrule
	 4                        & 22 & 28 & 34 & 33 & 23 & 35 & 23    & 17 & 16 & 15 & 14 & 13 & 12 & 12\\
	 5                        & 71 & 48 & 45 & 69 & 35 & 32 & 57    & 19 & 19 & 18 & 17 & 17 & 17 & 16\\
	 6                        & 73 & 71 & 70 & 46 & 69 & 57 &\hspace{-.5em}145    & 21 & 20 & 20 & 19 & 19 & 20 & 19\\
	 7                        &\hspace{-.2em}100 &\hspace{-.5em}106 & 86 & 71 & 92 & 67 & 61    & 22 & 22 & 21 & 21 & 21 & 22 & 22\\
	 8                        & 90 & 94 &\hspace{-.5em}127 & 98 &\hspace{-.5em}291 &\hspace{-.5em}106 & 73    & 23 & 23 & 22 & 22 & 22 & 22 & 23\\
	\bottomrule
	\end{tabular}
\end{center}
\caption{V-cycle with subspace corrected mass smoother for the L-shaped domain}
\label{tab:SCMS2}
\end{table}

\begin{table}[h]
\begin{center}
    \begin{tabular}{l|rrrrrrr|rrrrrrr}
    \toprule
						& \multicolumn{7}{c|}{(a) PCG -- Conforming}  & \multicolumn{7}{c}{(b) PCG -- Non-conforming}  \\    
    \midrule
    $L\,\backslash\, p$    &  2 &  3 &  4 &  5 &  6 &  7 &  8    &  2 &  3 &  4 &  5 &  6 &  7 &  8\\
    \midrule
	 3                        & 44 & 42 & 41 & 39 & 37 & 36 & 34    & 40 & 38 & 36 & 35 & 34 & 33 & 31\\
	 4                        & 48 & 47 & 45 & 43 & 43 & 40 & 41    & 44 & 44 & 42 & 42 & 40 & 40 & 39\\
	 5                        & 51 & 49 & 48 & 47 & 45 & 45 & 44    & 49 & 47 & 47 & 46 & 46 & 45 & 44\\
	 6                        & 52 & 51 & 49 & 48 & 47 & 46 & 45    & 58 & 57 & 57 & 56 & 55 & 54 & 53\\
	 7                        & 54 & 53 & 51 & 50 & 49 & 48 & 47    & 74 & 73 & 72 & 71 & 71 & 72 & 70\\
	\bottomrule
	\end{tabular}
\end{center}
\caption{V-cycle with subspace corrected mass smoother for the Yeti footprint}
\label{tab:SCMS3}
\end{table}

The PCG discretizations are presented in Tables~\ref{tab:SCMS2}~(b) and~\ref{tab:SCMS3}~(b);
we again obtain robustness in the grid size and the spline degree.
Here, for the Yeti footprint, we have again iteration counts that are slightly increasing
with the grid size; again, this observation can be explained by the fact that finer
grids allow to resolve the geometry function better, cf.~\eqref{eq:sgn}.

In principle, the method works also if the multigrid solver is applied directly, cf.
Table~\ref{tab:SCMS2}~(a). Here, we suffer from numerical instabilities which are amplified
with an increasing number of levels. One can avoid these instabilities, e.g., by increasing the number
of pre- and post smoothing steps. However, using the multigrid method as a preconditioner within
a PCG solver is obviously the more efficient approach.
 
\section{Multigrid with hybrid smoother}\label{sec:mg:hyb}

We have observed that a multigrid method with the subspace corrected mass smoother
is robust in the grid size and the spline degree and works well for both conforming
and discontinuous Galerkin discretizations. We have also observed that this
approach suffers from non-simple
geometry functions since it is based on the close connection between the stiffness matrix $A_\ell$
and the simplified stiffness matrix $\widehat{A}_\ell$. The results for the Gauss-Seidel smoother are different:
the  multigrid solver works badly both for large spline degrees and for
discontinuous Galerkin discretizations. However, by comparing Table~\ref{tab:GS1} with
Table~\ref{tab:GS2}, we observe that the method behaves quite robust in the geometry function.

Since the behavior of the two smothers is somewhat orthogonal, we can hope for a good method
if we combine them. Our idea is to use one forward Gauss-Seidel sweep followed by the
subspace corrected mass smoother for pre-smoothing and the 
subspace corrected mass smoother followed by one backward Gauss-Seidel sweep
for post-smoothing. The overall method is presented as Algorithm~\ref{alg:2}.

\begin{algorithm}[H]
\textsc{Multigrid}$\left(\ell, \ul{f}_\ell, \ul{u}_\ell\right)$\\
\mbox{}\qquad// Pre-Smoothing (forward Gauss-Seidel)\\
\mbox{}\qquad$\ul{u}_\ell\gets \ul{u}_\ell +  (L_\ell^{GS})^{-1}
                                    \left(\ul{f}_\ell-A_\ell\;\ul{u}_\ell\right)$ \\
\mbox{}\qquad// Pre-Smoothing (subspace corrected mass smoother)\\
\mbox{}\qquad\textbf{for} $n=1,\ldots,\nu_\ell$\\
\mbox{}\qquad\mbox{}\qquad $\ul{u}_\ell\gets \ul{u}_\ell +  (L_\ell^{SCMS})^{-1}
                                    \left(\ul{f}_\ell-A_\ell\;\ul{u}_\ell\right)$ \\
\mbox{}\qquad// Coarse-grid correction\\
\mbox{}\qquad\textbf{if} $\ell=1$\\
\mbox{}\qquad\mbox{}\qquad $\ul{u}_\ell\gets \ul{u}_\ell+ I_{\ell-1}^{\ell} A_{\ell-1}^{-1} I_{\ell}^{\ell-1}\left(\ul{f}_\ell - A_\ell
                      \;\ul{u}_\ell\right)$\quad// Direct solver\\
\mbox{}\qquad\textbf{else}\\
\mbox{}\qquad\mbox{}\qquad\textbf{for} $n=1,\ldots,\mu$\\
\mbox{}\qquad\mbox{}\qquad\mbox{}\qquad $\ul{u}_\ell\gets \ul{u}_\ell + I_{\ell-1}^{\ell}
			\textsc{Multigrid}\left(\ell-1, 
									I_{\ell}^{\ell-1}\left(\ul{f}_\ell - A_\ell
                      \;\ul{u}_\ell\right), 0\right)$
                      \\
\mbox{}\qquad// Post-Smoothing (subspace corrected mass smoother)\\
\mbox{}\qquad\textbf{for} $n=1,\ldots,\nu_\ell$\\
\mbox{}\qquad\mbox{}\qquad $\ul{u}_\ell\gets \ul{u}_\ell +  (L_\ell^{SCMS})^{-1}
                                    \left(\ul{f}_\ell-A_\ell\;\ul{u}_\ell\right)$ \\
\mbox{}\qquad// Post-Smoothing (backward Gauss-Seidel)\\
\mbox{}\qquad$\ul{u}_\ell\gets \ul{u}_\ell +  (L_\ell^{GS})^{-\top}
                                    \left(\ul{f}_\ell-A_\ell\;\ul{u}_\ell\right)$ \\
\mbox{}\qquad\textbf{return} $\ul{u}_\ell$
\caption{\label{alg:2}Multigrid algorithm with hybrid smoother}
\end{algorithm}

The convergence analysis from Section~\ref{sec:mg:scms} can be easily carried
over to the hybrid smoother. The
iteration matrix for the (V or W cycle) multigrid method with the hybrid smoother is given by
\[
	\widetilde{W}_\ell :=
		(I-(L^{GS}_\ell)^{-\top}A_\ell)
		W_\ell
		(I-(L^{GS}_\ell)^{-1}A_\ell),
\]
where $W_\ell$
is the iteration matrix of the (V or W cycle, respectively) multigrid method
with the subspace corrected mass smoother. Since the Gauss-Seidel iteration
is stable in the energy norm,  we obtain
\[\|\widetilde{W}_\ell\|_{A_\ell}\le 
\|I-(L^{GS}_\ell)^{-\top}A_\ell\|_{A_\ell}
\|W_\ell\|_{A_\ell}
\|I-(L^{GS}_\ell)^{-1}A_\ell\|_{A_\ell}
\le \|W_\ell\|_{A_\ell}.\]
So, we have using
the results from the last section the convergence of the 
W-cycle multigrid method with hybrid smoother. Thus, we obtain as follows.
\begin{corollary}
	Consider the multigrid solver with the hybrid smoother.
	Under the assumptions of Theorem~\ref{thrm:converg:conf} or
	\ref{thrm:converg}, respectively, 
	the W-cycle multigrid method converges with
	a convergence rate $q\le \max_\ell \nu_\ell^*/\nu_\ell$.
\end{corollary}

\begin{table}[h]
\begin{center}
    \begin{tabular}{l|rrrrrrr|rrrrrrr}
    \toprule
						& \multicolumn{7}{c|}{(a) Direct -- Conforming}  & \multicolumn{7}{c}{(b) PCG -- Conforming}  \\    
    \midrule
    $L\,\backslash\, p$    &  2 &  3 &  4 &  5 &  6 &  7 &  8    &  2 &  3 &  4 &  5 &  6 &  7 &  8\\
    \midrule
	 3                        & 12 & 15 & 22 & 29 & 35 & 40 & 47    &  9 & 11 & 14 & 17 & 20 & 21 & 22\\
	 4                        & 13 & 16 & 23 & 32 & 38 & 45 & 50    & 10 & 11 & 15 & 19 & 21 & 23 & 25\\
	 5                        & 14 & 16 & 24 & 35 & 41 & 47 & 53    & 11 & 11 & 16 & 19 & 22 & 24 & 26\\
	 6                        & 15 & 16 & 26 & 37 & 45 & 52 & 55    & 11 & 12 & 16 & 20 & 23 & 26 & 27\\
	 7                        & 16 & 17 & 26 & 38 & 47 & 54 & 57    & 12 & 12 & 17 & 21 & 24 & 26 & 28\\
	\bottomrule
	\end{tabular}
\end{center}
\caption{V-cycle with hybrid smoothing strategy for the Yeti footprint}
\label{tab:HYB1}
\end{table}

\begin{table}[h]
\begin{center}
    \begin{tabular}{l|rrrrrrr|rrrrrrr}
    \toprule
						& \multicolumn{7}{c|}{(a) Direct -- Non-conforming}  & \multicolumn{7}{c}{(b) PCG -- Non-conforming}  \\    
    \midrule
    $L\,\backslash\, p$    &  2 &  3 &  4 &  5 &  6 &  7 &  8    &  2 &  3 &  4 &  5 &  6 &  7 &  8\\
    \midrule
	 3                        & 23 & 18 & 35 & 34 & 37 & 24 & 28    & 15 & 17 & 18 & 19 & 20 & 20 & 21\\
	 4                        & 19 & 23 & 31 & 42 & 32 & 59 & 32    & 16 & 19 & 21 & 22 & 24 & 26 & 26\\
	 5                        & 20 & 25 & 34 & 50 & 46 & 67 & 62    & 17 & 20 & 22 & 25 & 27 & 29 & 31\\
	 6                        & 22 & 27 & 30 & 47 & 48 & 52 & 69    & 17 & 20 & 23 & 25 & 27 & 30 & 32\\
	 7                        & 22 & 29 & 33 & 40 & 58 & 52 & 65    & 17 & 21 & 23 & 26 & 29 & 31 & 33\\
	\bottomrule
	\end{tabular}
\end{center}
\caption{V-cycle with hybrid smoothing strategy for the Yeti footprint}
\label{tab:HYB2}
\end{table}

The tables for the experiments with the hybrid smoother are obtained with the following
code, where the values $L$ and $p$ are substituted accordingly:
\begin{center}
\begin{minipage}{.95\textwidth}
\begin{lstlisting}[mathescape,columns=flexible,basicstyle=\ttfamily]
> ./multiGrid_example -g domain2d/yeti_mp2.xml -r $L$ -p $p$
      -s hyb --MG.Damping .25 --MG.Scaling .1 -i d       $\tabnr{tab:HYB1}{a}$
> ./multiGrid_example -g domain2d/yeti_mp2.xml -r $L$ -p $p$ --DG
      --NonMatching -s hyb --MG.Damping .25 --MG.Scaling .1
      -i d                                               $\tabnr{tab:HYB2}{a}$
\end{lstlisting}
\end{minipage}
\end{center}
The results for the PCG experiments, presented in Tables~\ref{tab:HYB1}~(b) and~\ref{tab:HYB2}~(b),
are obtained by replacing the option {\tt -i d} by the option {\tt -i cg}.

For both cases, we obtain that the iteration counts are quite robust in the grid
size (even if the maximum number of iterations is not reached for the coarser grid
levels). We observe that the number of iterations
increases with the spline degree. This is indeed due to the fact that for small values
of $p$, the Gauss-Seidel smoother yields very fast convergence and that convergence
behavior is carried over to the hybrid smoother. For larger spline degrees, the
hybrid smoother's convergence behavior degrades mildly; this is due to the fact
that the Gauss-Seidel smoother is not completely capable to capture all effects
perfectly. Still, keeping in mind that the condition number of the stiffness matrix
grows exponentially with the spline degree, the observed behavior is still very
satisfactory.

Compared to applying the subspace corrected mass smoother only, the hybrid smoother
pays of in most cases. Certainly, applying the hybrid smoother with $\nu_\ell:=1$
means basically that $2$ pre- and $2$ post-smoothing steps are applied. Since the Gauss-Seidel 
smoother is slightly cheaper than the subspace corrected mass smoother, the costs for
one such cycle are smaller than the costs of two multigrid cycles with
the subspace corrected mass smoother only.

\begin{figure}[h]
\begin{center}
	\includegraphics[height=.28\textwidth]{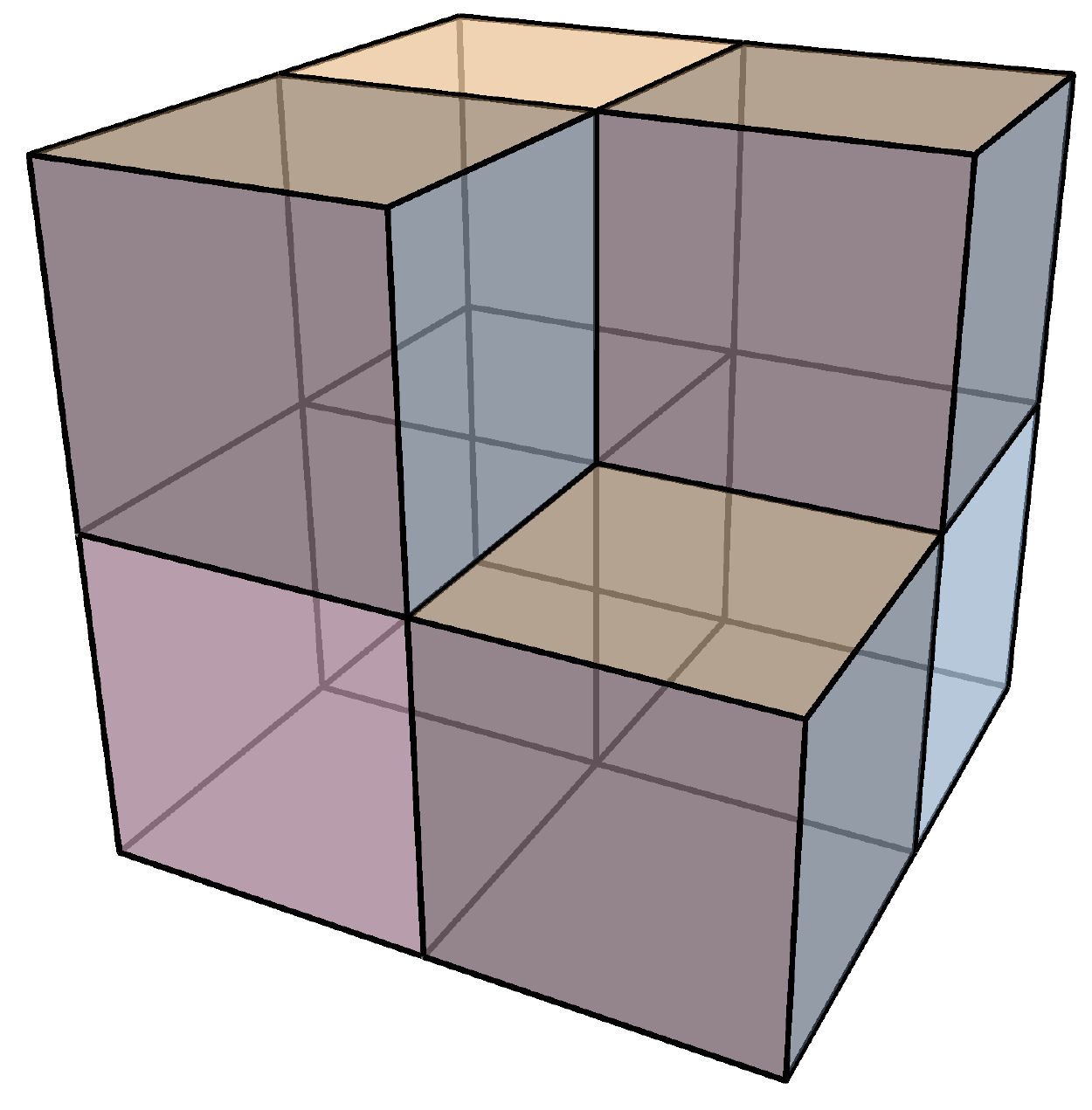} \qquad\qquad
	\includegraphics[height=.28\textwidth]{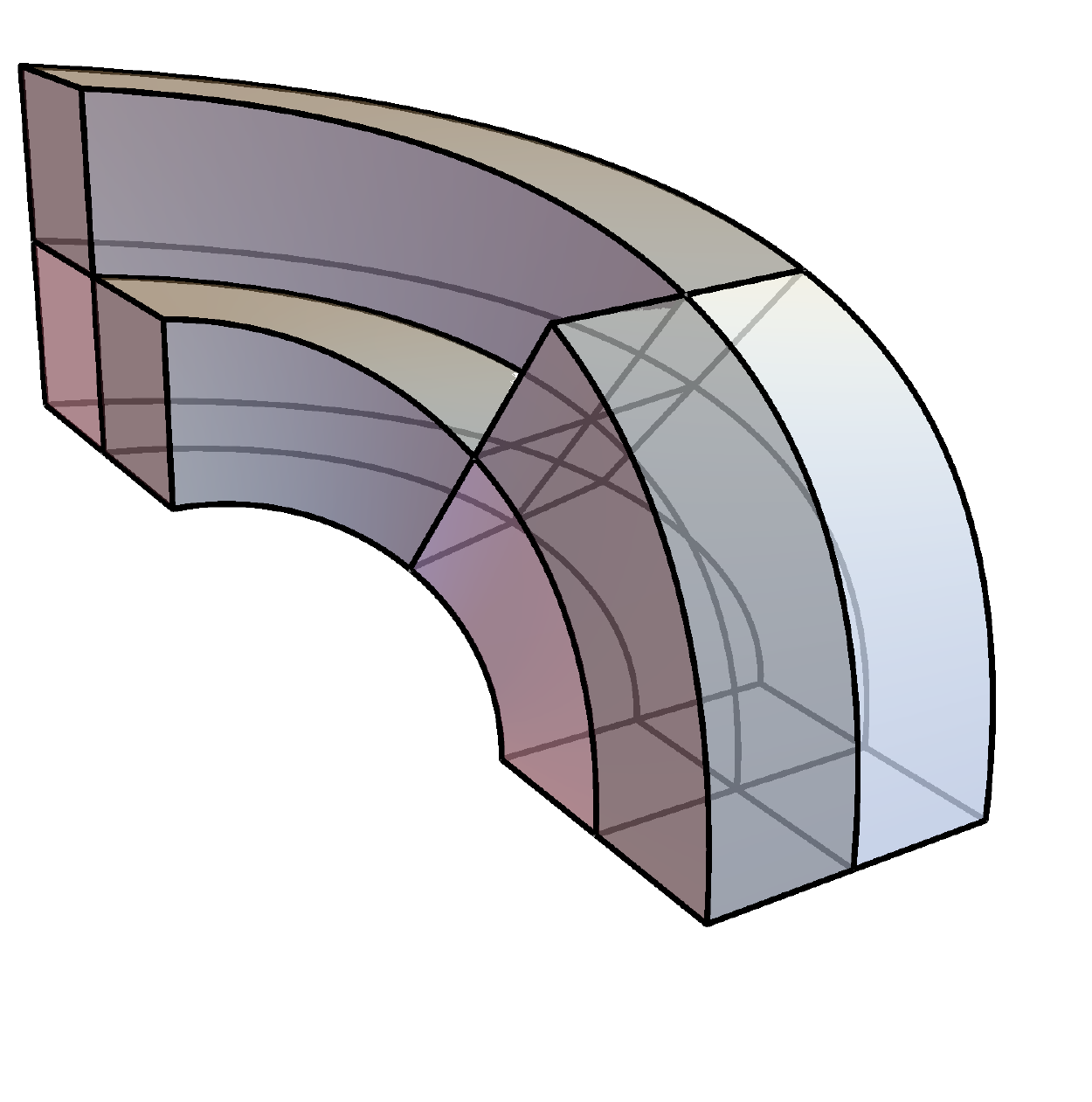}
\end{center}
\caption{The 3D computational domains: Fichera corner and twisted Fichera corner}
\label{fig:2}
\end{figure}

Besides the two-dimensional examples considered so far, the proposed methods can be directly
extended to three dimensional problems (even if the details of the convergence theory have not been worked out for these
cases). We consider the two domains depicted in Figure~\ref{fig:2}: the Fichera corner and
a variant of that domain with non-trivial geometry function, which we call twisted
Fichera corner.

The tables for the three dimensional domains are obtained with the following
code, where the values $L$ and $p$ are substituted accordingly:
\begin{center}
\begin{minipage}{.95\textwidth}
\begin{lstlisting}[mathescape,columns=flexible,basicstyle=\ttfamily]
> ./multiGrid_example -g domain2d/fichera.xml -r $L$ -p $p$
      -s hyb --MG.Scaling .12 --MG.Damping 1 -i cg        $\tabnr{tab:HYB3}{a}$
> ./multiGrid_example -g domain2d/twisted_fichera.xml -r $L$ -p $p$
      -s hyb --MG.Scaling .12 --MG.Damping .25 -i cg     $\tabnr{tab:HYB4}{a}$
\end{lstlisting}
\end{minipage}
\end{center}
The results for the DG experiments, presented in the Tables~\ref{tab:HYB3}~(b) and 
\ref{tab:HYB4}~(b), are obtained by adding the command line options
{\tt --DG --NonMatching}.

\begin{table}[h]
\begin{center}
    \begin{tabular}{l|rrrrr|rrrrr}
    \toprule
						& \multicolumn{5}{c|}{(a) PCG -- Conforming}  & \multicolumn{5}{c}{(b) PCG -- Non-conforming}  \\    
    \midrule
    $L\,\backslash\, p$    &  \quad2 &  \quad3 &  \quad4 &  \quad5 &  \quad6       &  \quad2 &  \quad3 &  \quad4 &  \quad5 &  \quad6 \\
    \midrule
	 2                        &  6 &  6 &  7 &  8 &  8     & 12 & 13 & 15 & 16 & 19 \\
	 3                        &  6 &  7 &  7 &  8 &  8     & 13 & 15 & 16 & 18 & 20 \\
	 4                        &  6 &  8 &  9 &  9 & 10     & 14 & 16 & 18 & 19 & 20 \\
	 5                        &  6 &  8 &  9 & 10 & 11     & 14 & 16 & 17 & 19 & 20 \\
	\bottomrule
	\end{tabular}
\end{center}
\caption{V-cycle with hybrid smoothing strategy for the Fichera corner}
\label{tab:HYB3}
\end{table}

\begin{table}[h]
\begin{center}
    \begin{tabular}{l|rrrrr|rrrrr}
    \toprule
						& \multicolumn{5}{c|}{(a) PCG -- Conforming}  & \multicolumn{5}{c}{(b) PCG -- Non-conforming}  \\    
    \midrule
    $L\,\backslash\, p$    &  \quad2 &  \quad3 &  \quad4 &  \quad5 &  \quad6       &  \quad2 &  \quad3 &  \quad4 &  \quad5 &  \quad6 \\
    \midrule
	 2                        & 10 & 13 & 14 & 15 & 17     & 22 & 25 & 28 & 30 & 31 \\
	 3                        & 13 & 15 & 18 & 20 & 22     & 29 & 31 & 33 & 36 & 38 \\
	 4                        & 14 & 17 & 19 & 22 & 25     & 31 & 34 & 37 & 40 & 42 \\
	 5                        & 16 & 17 & 20 & 23 & 26     & 31 & 36 & 40 & 44 & 47 \\
	\bottomrule
	\end{tabular}
\end{center}
\caption{V-cycle with hybrid smoothing strategy for the twisted Fichera corner}
\label{tab:HYB4}
\end{table}

Similar to the results for the Yeti footprint, Tables~\ref{tab:HYB3} and~\ref{tab:HYB4}
again show small iteration counts. For the twisted Fichera corner, we observe that the
number of iterations increases mildly when the grid gets refined; this is again to
be explained by the better resolution of the geometry function. Moreover, we observe
a very mild dependence on the spline degree.  

\section{Conclusions and outlook}\label{sec:fin}

We have presented robust multigrid solvers for multi-patch IgA
with conforming and non-conforming discretizations. We have given convergence results
that exactly state the robustness of the solvers in the grid size. Concerning the
dependence on the spline degree, the statements seem to be too pessimistic since
the solvers have been completely or (at least) rather robust in practice.

We have addressed another issue, which causes problems for all solvers that use the tensor-product
structure on the parameter domain: the dependence on the geometry function. We have proposed
a hybrid smoother between our subspace corrected mass smoother and the Gauss-Seidel smoother
which seems to reduce the effect on the geometry function. Finding approaches to better incorporate
the geometry function into the smoother itself seems to be an interesting topic for further research.

\section*{Appendix}

In the appendix, we give a proof of Theorem~\ref{thrm:converg} and of some auxiliary results.

Every constant $c$ used within the appendix is assumed to be independent of the
grid size, the grid level, the spline degree and the number of patches, but it may
depend on the geometry function (cf.~\cite[Assumption~3]{Takacs:2019}), the number
of neighbors of a patch (cf.~\cite[Assumption~2.3]{Takacs:2018}), the constant
in the elliptic regularity assumption (cf.~\cite[Assumption~3.1]{Takacs:2018})
and the quasi-uniformity of the grid, i.e., the ratio between the largest and the
smallest knot span of the knot vectors on one level.
We write $A \lesssim B$ if and only if there is a constant $c$ such that $A \le c \; B$ and
we write $A \eqsim B$ if and only if $A\lesssim B$ and $B\lesssim A$.

First, we show the following lemma, which is basically a trace inequality.

\begin{lemma}\label{lem:vertex}
		Let $S:=S^{(1)}\otimes S^{(2)}$ be the space of tensor-product
		splines of degree $p$ on a quasi-uniform grid with size $h$
		on the parameter domain $\widehat{\Omega}:=(0,1)^2$.
		Then, the estimate
		\[
			|u(0)|^2 \lesssim  \left(\log \left( 1 + \frac{p^4}{h^{2}\theta^{2}}\right)\right)^2
						\; \left( |u|_{H^1(\widehat{\Omega})}^2 + \theta^{2} 
								\|u\|_{L_2(\widehat{\Omega})}^2 \right)
		\]
		holds for all $u\in S$ and all $\theta \ge 1$.
\end{lemma}
\begin{proof}
		For $\nu=1,2$, let
		$(\psi_{\nu,i})_{i=1}^{N_\nu}$ be the eigenfunctions of $S^{(\nu)}$, i.e., such that
		\begin{align*}
				 & (\psi_{\nu,i},\psi_{\nu,j})_{L_2(0,1)}  = \delta_{i,j} \quad\mbox{and}\quad \\
				 & (\psi_{\nu,i}',\psi_{\nu,j}')_{L_2(0,1)} 
				 	+ \theta^2 (\psi_{\nu,i},\psi_{\nu,j})_{L_2(0,1)}  = \lambda_{\nu,i} \delta_{i,j},
		\end{align*}
		where $\delta_{i,j}$ is the Kronecker delta and
		$\lambda_{\nu,1} \le \lambda_{\nu,2} \le \cdots \le \lambda_{\nu,N_\nu}$
		are the corresponding
		eigenvalues.  Using coercivity of $(\cdot',\cdot')_{L_2(0,1)}$ and
		a standard inverse estimate, cf. \cite[Corollary~3.94]{Schwab:1998},
		we obtain
		\[
			\theta^2 \le \lambda_{\nu,1} \quad \mbox{and}\quad
			\lambda_{\nu,N_\nu} \lesssim p^4h^{-2}+\theta^2.
		\]
		We define level sets
		\[
			I_{\nu,m} := \{ i \; :\; \mu_{m-1}:=2^{m-1} \theta^2 \le \lambda_{\nu,i}
			< \mu_{m}:=2^m \theta^2 \}
		\]
		for $m\in\{1,2,3,\ldots,M\}$,
		where
		\[
			M:=1+\max_{\nu\in\{1,2\}} \lfloor \log_2(\theta^{-2} \lambda_{\nu,N_\nu})\rfloor
					\lesssim \log ( 1 + p^4h^{-2}\theta^{-2})  
		\]
		is the number of level sets.
		Note that by construction every eigenvalue belongs to exactly one level set.
		Every function $u\in S$ can be represented as
		\begin{align*}
			u(x,y) &= \sum_{i=1}^{N_1} \sum_{j=1}^{N_2} u_{i,j} \psi_{1,i}(x)\psi_{2,j}(y) 
			= \sum_{m=1}^{M} \sum_{n=1}^{M} \underbrace{ \sum_{i\in I_{1,m}}
			\sum_{j\in I_{2,n}} u_{i,j} \psi_{1,i}(x)\psi_{2,j}(y) } _ { \displaystyle w_{m,n}(x,y):= }.
		\end{align*}
		A standard trace estimate, cf.~\cite[Lemma~4.4]{Takacs:2018}, yields
		\begin{equation}\label{eq:trace}
		\begin{aligned}
			|w_{m,n}(0)|^2 & \lesssim \|w_{m,n}\|_{L_2(\{0\}\times(0,1))}\|w_{m,n}\|_{H^1(\{0\}\times(0,1))}\\
			& \lesssim \|w_{m,n}\|_{0,0,1}^{1/2} \|w_{m,n}\|_{1,0,1} ^{1/2}
			\|w_{m,n}\|_{0,1,1} ^{1/2}\|w_{m,n}\|_{1,1,1} ^{1/2}
		\end{aligned}
		\end{equation}
		where
		\[
		\|w\|_{a,b,\eta}^2 
		:= \left\|\frac{\partial^{a+b}}{\partial x^a\partial x^b} w\right\|_{L_2(\widehat{\Omega})}^2
		+ \eta^{2(a+b)} \|w\|_{L_2(\widehat{\Omega})}^2
		\]
		for $a,b\in\mathbb{N}_0$.
		Since $\theta \ge 1$ and since all eigenvalues are
		in $I_{1,m}$ or $I_{2,n}$, respectively, we obtain 
		\[
			\|w\|_{a,b,1}^2 \le
			\|w\|_{a,b,\theta}^2
				\eqsim \mu_m^{a}\mu_n^{b} \|w_{m,n}\|_{L_2(\widehat{\Omega})}^2.
		\]
		Using~\eqref{eq:trace}, we obtain further
		\begin{align*}
			|w_{m,n}(0)|^2 
			& \lesssim
				\mu_m\mu_n\|w_{m,n}\|_{L_2(\widehat{\Omega})}^2 
			 \eqsim \|w_{m,n}\|_{1,0,\theta} \|w_{m,n}\|_{0,1,\theta} \\
			&	\le \left( |w_{m,n}|_{H^1(\widehat{\Omega})}^2 + \theta^2 \|w_{m,n}\|_{L_2(\widehat{\Omega})}^2\right).
		\end{align*}
		Finally, the Cauchy-Schwarz inequality and orthogonality of the basis functions
		(both in $L_2$ and $H^1$) yield
		\begin{align*}
			|u(0)|^2 
				& \lesssim M^2 \sum_{m=1}^M \sum_{n=1}^M |w_{m,n}(0)|^2 \\
				& \lesssim M^2 \sum_{m=1}^M \sum_{n=1}^M \left(|w_{m,n}|_{H^1(\widehat{\Omega})}+
				\theta^2 \|w_{m,n}\|_{L_2(\widehat{\Omega})} \right) ^2 \\
				& = M^2 \left( |u|_{H^1(\widehat{\Omega})}^2 + \theta^2
					\|u\|_{L_2(\widehat{\Omega})}^2 \right),
		\end{align*}
		which finishes the proof.
\qed\end{proof}

Now, we give bounds on the smoother which allow to show the smoothing property.

\begin{lemma}\label{lem:smp1}
	Provided the assumptions of Theorem~\ref{thrm:converg}, the estimate
	\[
		A_\ell \le L_\ell \lesssim 
		p (\log p)^2  (1+L-\ell)^2 2^{L-\ell} \frac{\tau^*}{\tau} \frac{\delta^*}{\delta}
		\widetilde{L}_\ell 
	\]
	holds, where 
	$\widetilde{L}_\ell:=  Q_\ell + (1+2^{\ell-L}h_\ell^{-2})  M_\ell$ and
	$M_\ell$ is the standard mass matrix.
\end{lemma}
\begin{proof}
	The proof of this Lemma requires the notation from~\cite{Takacs:2018}, i.e.,
	we denote the set of all patch-interiors by $\mathbb{K}$, the set of all edges
	by $\mathbb{E}$ and the set of all vertices by $\mathbb{V}$.
	
	Observe that we have
	\begin{equation}\label{eq:equiv}
				A_\ell \eqsim Q_\ell
	\end{equation}
	for all $\ell=0,2,\ldots,L$, where $Q_\ell$ is defined by~\eqref{eq:model:discr2}
	and~\eqref{eq:a:galerkin}. For $\ell = L$, this statement directly
	follows from~\cite[Theorem~8]{Takacs:2019}.
	Since \cite[Theorem~8]{Takacs:2019} also holds in cases of over-penalization, we can
	apply that theorem also to the case $\ell<L$ and obtain~\eqref{eq:equiv} also in that
	case.
	
	As first step, we bound $L_\ell$ from below. (The following arguments are analogous
	to~\cite[Lemma~4.3]{Takacs:2018}.) The triangle inequality yields
	\begin{equation}\label{eq:19a}
		 A_\ell
		 	\lesssim \sum_{T\in \mathbb{K}\cup\mathbb{E}\cup\mathbb{V}} P_{\ell,T} (P_{\ell,T}^\top A_{\ell}P_{\ell,T}^\top) P_{\ell,T}^\top.
	\end{equation}
	For $T\in \mathbb{K}$, \cite[Lemma~8]{Hofreither:Takacs:2017} and~\eqref{eq:equiv}
	yield $L_{\ell,T}\gtrsim P_{\ell,T}^\top Q_{\ell}P_{\ell,T}^\top \gtrsim
	P_{\ell,T}^\top A_{\ell}P_{\ell,T}^\top$. For
	$T\in \mathbb{E}\cup\mathbb{V}$, we have by definition
	$L_{\ell,T} = P_{\ell,T}^\top A_{\ell}P_{\ell,T}^\top$. Thus, we obtain from~\eqref{eq:19a}
	\begin{equation}\nonumber
		 A_\ell
		 	\lesssim
		 	\sum_{T\in \mathbb{K}\cup\mathbb{E}\cup\mathbb{V}} P_{\ell,T} L_{\ell,T} P_{\ell,T}^\top
	\end{equation}
	and for all $\tau\in (0,\tau^*)$ with $\tau^*$ small enough further
	\begin{equation}\nonumber
		 A_\ell
		 	\le
		 	\tau^{-1}
		 	\sum_{T\in \mathbb{K}\cup\mathbb{E}\cup\mathbb{V}} P_{\ell,T} L_{\ell,T} P_{\ell,T}^\top
		 	= L_\ell,
	\end{equation}
	which shows the first part of the desired inequality.

	Now, we bound $L_\ell$ from above.
	We use the decomposition
	\[
			Q_\ell = K_\ell + \frac{\sigma p^2}{h_L} J_\ell,
	\]
	cf.~\eqref{eq:model:discr2}.
	Using~\eqref{eq:equiv}, we obtain
	\[
			L_{\ell,T}
				= P_{\ell,T}^\top A_\ell P_{\ell,T}
				\eqsim P_{\ell,T}^\top Q_\ell P_{\ell,T}
				= \underbrace{
						P_{\ell,T}^\top K_\ell P_{\ell,T}
				  }_{\displaystyle \widetilde K_{\ell,T} := }
				 + \frac{\sigma p^2}{h_L} \underbrace{
				 	     P_{\ell,T}^\top J_\ell P_{\ell,T}
				  }_{\displaystyle \widetilde J_{\ell,T} := }
	\]
	for all $T\in \mathbb{E}\cup \mathbb{V}$
	and, therefore,
	\begin{equation}\label{eq:upper:decomp}
			L_\ell \eqsim \widetilde K_{\ell}
				+ \frac{\sigma p^2}{h_L}
				\widetilde J_{\ell},
	\end{equation}
	where
	\[
		 \widetilde K_{\ell} :=
				\tau^{-1}
					\sum_{T \in \mathbb{K}} P_{\ell,T} L_{\ell,T} P_{\ell,T}^\top
					+ \tau^{-1}\sum_{T \in \mathbb{E}\cup \mathbb{V}} P_{\ell,T} K_{\ell,T} P_{\ell,T}^\top
	\]
	and
	\[
			\widetilde J_{\ell} :=\tau^{-1} \sum_{T \in \mathbb{E}\cup \mathbb{V}} P_{\ell,T} J_{\ell,T} P_{\ell,T}^\top.
	\]
	Completely analogous to \cite[Lemma~4.7]{Takacs:2018}, we obtain 
	\begin{equation}\label{eq:upper1}
		\widetilde K_{\ell}
		 \lesssim  p \frac{\tau^*}{\tau}\frac{\delta^*}{\delta} (K_\ell + h_\ell^{-2} M_\ell)
		\le p \frac{\tau^*}{\tau}\frac{\delta^*}{\delta}  (Q_\ell + h_\ell^{-2} M_\ell)
	    \le p 2^{L-\ell} \frac{\tau^*}{\tau}\frac{\delta^*}{\delta}  \widetilde{L}_\ell.
	\end{equation}
	
	So, it remains to bound $\widetilde J_{\ell}$ from above.
	Since the restriction of $J_\ell$ to any patch-interior vanishes,
	the same arguments as in the proof of~\cite[Lemma~4.7]{Takacs:2018}
	and the triangle inequality yield
	\begin{equation}\nonumber
	\begin{aligned}
		\sum_{T\in \mathbb E} \| P_TP_T^\top \ul{u}_\ell \|_{J_\ell}^2
			&=  \left\| \sum_{T\in \mathbb E}P_TP_T^\top \ul{u}_\ell \right\|_{J_\ell}^2
			=  \left\| \ul{u}_\ell - \sum_{T\in \mathbb V}P_TP_T^\top \ul{u}_\ell \right\|_{J_\ell}^2 \\
			& \le  \| \ul{u}_\ell \|_{J_\ell}^2 + \sum_{T\in \mathbb V} \| P_TP_T^\top \ul{u}_\ell \|_{J_\ell}^2.
	\end{aligned}
	\end{equation}
	and therefore
	\begin{equation}\label{eq:Tsum}
			 \sum_{T \in \mathbb{E}\cup \mathbb{V}} P_{\ell,T} \widetilde J_{\ell,T} P_{\ell,T}^\top
			\lesssim
			J_\ell +
			  \sum_{T \in \mathbb{V}} P_{\ell,T} \widetilde J_{\ell,T} P_{\ell,T}^\top.
	\end{equation}
	Note that $J_\ell$ models jumps and note that these jumps can be bounded from
	above using the triangle inequality with the function values on both sides. Thus,
	we obtain
	\[
		\sum_{T\in \mathbb V}
				\| P_TP_T^\top \ul{u}_\ell \|_{J_\ell}^2
		\lesssim
		\sum_{k=1}^K
		\sum_{T\in \mathbb V}
				( u_{\ell}|_{\Omega_k}|_T )^2
				\| \psi \|_{L_2(0,1)}^2
				,
	\]
	where $u_{\ell}|_{\Omega_k}$ is the restriction of $u_\ell$ to the patch $\Omega_k$
	and $ u_{\ell}|_{\Omega_k}|_T$ is the evaluation of the continuous extension of
	that function to the vertex $T$ at that vertex and
	$\psi(x)=\max\{1-x/h_\ell,0\}^p$ is the corresponding basis function. Using
	$
			\| \psi \|_{L_2(0,1)}^2 \eqsim p^{-1} h_\ell
	$,
	cf.~\cite[Eq.~(4.16)]{Takacs:2018},
	Lemma~\ref{lem:vertex} (with $\theta:=(1+h_\ell^{-2}2^{\ell-L})^{1/2}$),
	and $h_\ell \le 1$,
	we further obtain
	\begin{align*}
		&\sum_{T\in \mathbb V}
				\| P_TP_T^\top \ul{u}_\ell \|_{J_\ell}^2 
		 \lesssim
		 \frac{h_\ell}{p}
		\left(\log\left(1+\frac{p^4}{h_\ell^{2}(1+h_\ell^{-2}2^{\ell-L})^{2}} \right)\right)^2 
		\\ &\qquad\qquad \sum_{k=1}^K
			\left( |u_\ell \circ G_k |_{H^1(\widehat{\Omega})}^2 + (1+ 2^{\ell-L} h_\ell^{-2})
					 \|u_\ell \circ G_k \|_{L_2(\widehat{\Omega})}^2\right) \\
		& \lesssim
		\frac{h_\ell}{p}
		(\log p) (1+L-\ell)^2
		\sum_{k=1}^K
			\left( |u_\ell \circ G_k |_{H^1(\widehat{\Omega})}^2 + (1+ 2^{\ell-L} h_\ell^{-2})
					 \|u_\ell \circ G_k \|_{L_2(\widehat{\Omega})}^2 \right).
	\end{align*}
	Using~\cite[Lemma~6]{Takacs:2019}, we obtain
	\[
		\sum_{T\in \mathbb V}
				\| P_TP_T^\top \ul{u}_\ell \|_{J_\ell}^2
		\lesssim
		\frac{h_\ell}{p}
		(\log p)^2 (1+L-\ell)^2 
			\left(|u_\ell|_{H^1(\Omega)}^2+ (1+ 2^{\ell-L} h_\ell^{-2})  \|u_\ell\|_{L_2(\Omega)}^2\right) 
	\]
	and thus
	\[
			\sum_{T \in \mathbb{V}} P_{\ell,T} \widetilde J_{\ell,T} P_{\ell,T}^\top
			\lesssim \frac{h_\ell}{p} (\log p)^2 (1+L-\ell)^2 
			( K_\ell + (1+ 2^{\ell-L} h_\ell^{-2})   M_\ell ).
	\]
	This shows together with~\eqref{eq:Tsum} and $h_L \eqsim 2^{\ell-L} h_\ell$
	\begin{align*}
		\frac{\sigma p^2}{h_L} \widetilde J_{\ell}
		&=
			\tau^{-1} \frac{\sigma p^2}{h_L}
			\sum_{T \in \mathbb{E} \cup \mathbb{V}} P_{\ell,T} \widetilde J_{\ell,T} P_{\ell,T}^\top\\
		&\lesssim
			\tau^{-1} \left( \frac{\sigma p^2}{h_L}
				J_\ell + 
					  p (\log p)^2 (1+L-\ell)^2 2^{L-\ell} 
						( K_\ell + (1+ 2^{\ell-L} h_\ell^{-2})   M_\ell ) \right)\\
		&\lesssim
			\tau^{-1}
			p (\log p)^2 (1+L-\ell)^2 2^{L-\ell}  \widetilde{L}_\ell.
	\end{align*}
	Since $\delta \in (0,\delta^*)$ and since $\tau^* \eqsim 1$, we obtain
	\begin{align*}
		\frac{\sigma p^2}{h_L} \widetilde J_{\ell}
		&\lesssim
			p (\log p)^2 (1+L-\ell)^2 2^{L-\ell}
			\frac{\tau^*}{\tau}\frac{\delta^*}{\delta}
   			\widetilde{L}_\ell,
	\end{align*}
	which finishes together with~\eqref{eq:upper:decomp}
	and~\eqref{eq:upper1} the proof.
\qed\end{proof}

\begin{lemma}\label{lem:stab}
	Let $ \|v\|_{Q_\ell^+}^2 := |v|_{H^1(\Omega)}^2 + \sigma^{-2} p^{-4} 4^{\ell-L} h_\ell^2 
          |v|_{H^2(\Omega)}^2$.
	The estimate
	\[
		\inf_{v_\ell \in V_\ell^n} \|u-v_\ell\|_{Q_\ell^+}^2 \lesssim |u|_{H^2(\Omega)}^2
	\]
	holds for all $u \in H^2(\Omega)$.
\end{lemma}
\begin{proof}
	Let
	\[
		W:=\{
				u \in H^1(\Omega) \;:\;
					u\circ G_k \in S_{1,1}(\widehat{\Omega})
					\mbox{ for all } k = 1,\ldots,K
		\},
	\]
	be the set of all globally continuous functions which are linear locally.
	Observe that $W\subseteq V_\ell^n$. Using $u$ and $w$ being continuous, we obtain
	\begin{align*}
		\|u-w\|_{Q_\ell^+}^2 = |u-w|_{H^1(\Omega)}^2 + 
				\frac{h_\ell^2}{\sigma^2 p^4 4^{L-\ell}} |u-w|_{H^2(\Omega)}^2.
	\end{align*}
	For the choice
	$w\in H^1(\Omega)$ with $w|_{\Omega_k} := w_k = \widehat{w}_k\circ G_k^{-1}$,
	where
	\[
			\widehat{w}_k(x,y) :=
				\sum_{i=0}^1 \sum_{j=0}^1
				\widehat{\phi}_i(x)\phi_j(y) \widehat{u}_k(i,j) \quad \mbox{and}\quad
				\widehat{\phi}_0(t):=1-t \quad \mbox{and}\quad \widehat{\phi}_1(t):=t,
	\]
	we further obtain using
	standard approximation error estimates and \cite[Lemma~6]{Takacs:2019}
	\[
		\inf_{v_\ell \in V_\ell^n} \|u-v_\ell\|_{Q_\ell^+}^2
		\le \|u-w\|_{Q_\ell^+}^2 
			\lesssim \left(1+\frac{h_\ell^2}{\sigma^2 p^4 4^{L-\ell}} \right)|u|_{H^2(\Omega)}^2.
	\]
	Using $h_\ell \le 1$, $\sigma \ge 1$, $p\ge 2$, and $L\ge \ell$, we obtain the desired result.
\end{proof}

\begin{lemma}\label{lem:approx}
	Provided the assumptions of Theorem~\ref{thrm:converg}, the estimate
	\[
		\|  (I-I_{\ell-1}^\ell A_{\ell-1}^{-1}I_\ell^{\ell-1} A_\ell )
					A_\ell^{-1} \widetilde{L}_\ell  \|_{\widetilde{L}_\ell}  \lesssim
					(\log p)^2 
	\]
	holds, where $\widetilde{L}_\ell$ is as in Lemma~\ref{lem:smp1}.
\end{lemma}
\begin{proof}
	Let $u_\ell \in V_\ell^n$ be arbitrary but fixed. Let $f_\ell \in V_\ell^n$ be such that
	\[
		(u_\ell,v_\ell)_{A_L}
				= (f_\ell,v_\ell)_{L_2(\Omega)} \quad \mbox{for all} \quad v_\ell \in V_\ell^n.
	\]
	Let $u_{\ell-1} \in V_{\ell-1}$ and $u\in V$ be such that
	\begin{align*}
		(u_{\ell-1},v_{\ell-1})_{A_L}& = (f_\ell,v_{\ell-1})_{L_2(\Omega)} \quad \mbox{for all} \quad v_{\ell-1} \in V_{\ell-1} , \\
		(\nabla u, \nabla v)_{L_2(\Omega)} &= (f_\ell,v)_{L_2(\Omega)} \quad \mbox{for all} \quad v \in V.
	\end{align*}
	Using $f_\ell\in L_2(\Omega)$ and full elliptic
	regularity, cf.~\cite[Assumption~3.1]{Takacs:2018}, we obtain $u\in H^2(\Omega)$ and
	\begin{equation}\label{eq:above}
		|u|_{H^2(\Omega)} \lesssim \|f_\ell\|_{L_2(\Omega)} 
		= \sup_{v_\ell \in V_\ell^n} \frac{(f_\ell,v_\ell)_{L_2(\Omega)}}{\|v_\ell\|_{L_2(\Omega)}}
		= \sup_{v_\ell \in V_\ell^n} \frac{(u_\ell,v_\ell)_{A_L}}{\|v_\ell\|_{L_2(\Omega)}}
				= \| \ul{u}_\ell \|_{A_\ell M_\ell^{-1} A_\ell}.
	\end{equation}
	\cite[Theorems~12 and 13]{Takacs:2019} and Lemma~\ref{lem:stab} yield
	\begin{align*}
			\| u-u_\ell \|_{Q_\ell}^2
				& \lesssim  \min\{1,(\log \sigma_\ell)^2 \sigma_\ell^{1/(2p-1)} h_\ell^2\} |u|_{H^2(\Omega)}^2, \\
			\| u-u_{\ell-1} \|_{Q_\ell}^2
				& \lesssim  \min\{1, (\log \sigma_{\ell-1})^2 \sigma_{\ell-1}^{1/(2p-1)} h_{\ell-1}^2\} |u|_{H^2(\Omega)}^2,
	\end{align*}
	where
	$
				\sigma_\ell = 2^{L-\ell} p^2 \sigma.
	$
	Using the triangle inequality, $\sigma \eqsim 1$, $\log (ab) \lesssim \log a \log b$,
	$h_\ell \eqsim h_{\ell-1}$, $p\ge2$, and we obtain further
	\begin{align*}
		&\| u_\ell-u_{\ell-1} \|_{Q_\ell}^2
			\lesssim
					\min\{1,(\log \sigma_\ell)^2 \sigma_\ell^{1/(2p-1)} h_\ell^2\} |u|_{H^2(\Omega)}^2\\
			&\qquad	\lesssim 
					\min\{1,(\log p)^2  p^{2/(2p-1)}(1+L-\ell)^2 2^{(L-\ell)(1/(2p-1))} h_\ell^2 \}
					|u|_{H^2(\Omega)}^2 \\
			&\qquad\lesssim 
					\min \{1, (\log p)^2 2^{L-\ell}  h_\ell^2 \} |u|_{H^2(\Omega)}^2 
			 \lesssim
			 		(\log p)^2 (1+2^{\ell-L}h_\ell^{-2})^{-1} |u|_{H^2(\Omega)}^2
	\end{align*}
	Using~\eqref{eq:equiv},
	\eqref{eq:above} and the definition of $u_{\ell-1}$, we obtain further
	\[
		\| ( I - I_{\ell-1}^\ell A_{\ell-1}^{-1} I_\ell^{\ell-1} A_\ell ) \ul{u}_\ell \|_{A_\ell}^2
			\lesssim 
				(\log p)^2 (1+2^{\ell-L}h_\ell^{-2})^{-1}
					\| \ul{u}_\ell \|_{A_\ell M_\ell^{-1} A_\ell}^2.
	\]
	This yields
	\[
		\| A_\ell^{1/2}  ( I - I_{\ell-1}^\ell A_{\ell-1}^{-1} I_\ell^{\ell-1} A_\ell ) 
					A_\ell^{-1} M_\ell^{1/2} \|^2
		\lesssim
			(\log p)^2 (1+2^{\ell-L}h_\ell^{-2})^{-1}
	\]
	and thus
	\[
		\| ( I - I_{\ell-1}^\ell A_{\ell-1}^{-1} I_\ell^{\ell-1} A_\ell ) \ul{u}_\ell 
		\|_{(1+2^{\ell-L}h_\ell^{-2})  M_\ell}^2 
		\lesssim (\log p)^2 
		  \|\ul{u}_\ell\|_{A_\ell}\quad\mbox{for all}\quad \ul{u}_\ell
		\in \mathbb{R}^{N_\ell}.
	\]
	Using this estimate and the stability of the $A_\ell$-orthogonal projection, we obtain
	\begin{align*} 
		\| ( I - I_{\ell-1}^\ell A_{\ell-1}^{-1} I_\ell^{\ell-1} A_\ell ) \ul{u}_\ell 
				\|_{\widetilde{L}_\ell}^2 
		\lesssim  (\log p)^{2} \|\ul{u}_\ell\|_{A_\ell}\qquad\mbox{for all}\quad \ul{u}_\ell\in \mathbb{R}^{N_\ell}
	\end{align*}
	and further
	\[
		\| \widetilde{L}_\ell^{1/2}  ( I - P A_{\ell-1}^{-1} P^\top A_\ell ) A_\ell^{-1} \widetilde{L}_\ell^{-1/2} \|
			 \lesssim \log p.
	\]
	Using the identity $\|A^\top A\|\le \|A\|^2$, we finally obtain the desired result.
\qed\end{proof}

Finally, we can show Theorem~\ref{thrm:converg}. Here, we follow the classical approach
as introduced by Hackbusch, cf.~\cite{Hackbusch:1985}.
\begin{proof}[of Theorem~\ref{thrm:converg}]
	Lemma~\ref{lem:smp1} yields $A_\ell\le L_\ell$.
	Using standard arguments, cf.~\cite[Lemma~2]{Hofreither:Takacs:Zulehner:2017}
	or \cite{Hackbusch:1985},	the smoothing property
	\[
			\|L_\ell^{-1} A_\ell (I- L_\ell^{-1} A_\ell)^\nu  \|_{L_\ell} \le \frac{1}{\nu+1}\le \frac{1}{\nu  }
	\]
	follows. Using Lemma~\ref{lem:smp1}, we obtain further
	\[
			\|\widetilde{L}_\ell^{-1} A_\ell (I- L_\ell^{-1} A_\ell)^\nu  \|_{\widetilde{L}_\ell} \lesssim \frac{p (\log p)^2 (1+L-\ell)^2 2^{L-\ell} }{\nu}
			\frac{\tau^*}{\tau}
			\frac{\delta^*}{\delta}
			,			
	\]
	which shows together Lemma~\ref{lem:approx}
	\begin{align*}
		&\|(I-L_\ell^{-1} A_\ell)^\nu  ( I - P A_{\ell-1}^{-1} P^\top A_\ell ) (I-L_\ell^{-1} A_\ell)^\nu  \|_{A_\ell} \\
		&\qquad \le
		\|( I - P A_{\ell-1}^{-1} P^\top A_\ell ) (I-L_\ell^{-1} A_\ell)^\nu  \|_{A_\ell} \\
		&\qquad \le
		\|( I - P A_{\ell-1}^{-1} P^\top A_\ell ) (I-L_\ell^{-1} A_\ell)^\nu  \|_{\widetilde{L}_\ell}\\
		&\qquad \le
		\|( I - P A_{\ell-1}^{-1} P^\top A_\ell )A_\ell^{-1}\widetilde{L}_\ell\|_{\widetilde{L}_\ell}
		\| \widetilde{L}_\ell^{-1}A_\ell (I-L_\ell^{-1} A_\ell)^\nu  \|_{\widetilde{L}_\ell}\\
		&\qquad \lesssim \frac{p (\log p)^4 (1+L-\ell)^2 2^{L-\ell}}{\nu }
			\frac{\tau^*}{\tau}
			\frac{\delta^*}{\delta} .
	\end{align*}
	This statement
	shows convergence of the two-grid method if $\nu$ is large enough.
	Standard arguments, cf.~\cite{Hackbusch:1985}, allow to extend the
	analysis to the W-cycle multigrid method.
\end{proof}

\section*{Acknowledgments}

The author was supported by the Austrian Science Fund (FWF):
grant P31048, and by the bilateral project DNTS-Austria 01/3/2017
(WTZ BG 03/2017), funded by Bulgarian National Science Fund and OeAD (Austria).

 
\bibliographystyle{amsplain}
\bibliography{references}

\end{document}